\documentclass[10pt]{amsart}
\usepackage{qtree,array,youngtab}
\usepackage{color, graphicx, comment}
\usepackage{tikz}

\usepackage{mathtools}
\DeclarePairedDelimiter\ceil{\lceil}{\rceil}
\DeclarePairedDelimiter\floor{\lfloor}{\rfloor}

\definecolor{darkgreen}{rgb}{0.,0.5,0.}

\newcommand{\la}{\lambda}

\allowdisplaybreaks

\numberwithin{equation}{section} \overfullrule 5pt
\newtheorem{thm}{Theorem}[section]

\newtheorem{lem}[thm]{Lemma}

\theoremstyle{definition}

\newcommand{\PP}[1]{\text{\rm #1}}

\newcommand{\inv}{\text{inv}}

\title[Asymptotic formulas and their applications]{%
Some useful theorems for asymptotic formulas and their applications to skew plane partitions and cylindric partitions}
\date{July 11, 2017}

\author{Guo-Niu HAN and Huan XIONG}
\address{Universit\'e de Strasbourg, CNRS, IRMA UMR 7501, F-67000 Strasbourg, France}
\email{guoniu.han@unistra.fr, \quad xiong@math.unistra.fr}

\subjclass[2010]{05A16, 05A17}

\keywords{asymptotic formula, integer partition, plane partition, 
cylindric partition}

\begin{document}
\begin{abstract} 
Inspired by the works of Dewar, Murty and Kot\v{e}\v{s}ovec, we establish some useful theorems for asymptotic formulas. As an application, we obtain asymptotic formulas for the numbers of skew plane partitions and cylindric partitions. We prove that the order of the asymptotic formula for the skew plane partitions of fixed width depends only on the width of the region, not on the profile (the skew zone) itself, while this is not true for cylindric partitions.
\end{abstract}

\maketitle

\section{Introduction}\label{sec:Asym} 
\medskip

Inspired by the works of Dewar, Murty and Kot\v{e}\v{s}ovec \cite{DewarMurty2013, Kotesovec2015}, we establish some useful theorems for
asymptotic formulas.
Define
\begin{equation}\label{def:psi}
\psi_{n}(v, r, b; p) := v \sqrt{\frac{p(1-p)}{2\pi}} \frac{r^{b+(1-p)/2}}{n^{b+1-p/2}} \exp({n^p}{r^{1-p}})
\end{equation}
for $n\in \mathbb{N},$ $v,b\in \mathbb{R},$ $r>0,$ $0<p<1$.

\begin{thm}\label{th:asy_main2}
Let $t_1$ and $t_2$ be two given positive integers with $gcd(t_1,t_2)=1$. 
	Suppose that
	\begin{equation*}
		F_1(q) = \sum_{n=0}^\infty a_{t_1n} q^{t_1n} \qquad \text{and}\qquad
		F_2(q) = \sum_{n=0}^\infty c_{t_2n} q^{t_2n}
	\end{equation*}
	are two power series such that their coefficients satisfy the 
	asymptotic formulas
	\begin{align}
		a_{t_1n} &\ \sim\  t_1\psi_{t_1n}(v_1, r_1, b_1; p),\label{eq:asym_a} \\
		c_{t_2n} &\ \sim\  t_2\psi_{t_2n}(v_2, r_2, b_2; p),\label{eq:asym_c}
	\end{align}
	where   $v_1,b_1,v_2, b_2\in \mathbb{R}$, $r_1,r_2>0$, $0<p<1$. 
Then, the coefficients $d_n$ in the product
	$$F_1(q)F_2(q) = \sum_{n=0}^\infty d_n q^{n}$$ 
	satisfy the following asymptotic formula
	\begin{align}\label{eq:asy_d_n}
		d_n \sim 
		\psi_{n}(v_1v_2, r_1+r_2, b_1+b_2 ; p).
	\end{align}
\end{thm}

Some special cases of Theorem \ref{th:asy_main2} have been established.  
In 2013, Dewar and Murty \cite{DewarMurty2013} 
proved the case of $p=1/2,\, t_1=t_2=1$.
Later, Kot\v{e}\v{s}ovec \cite{Kotesovec2015} obtained the case of $0<p<1$,  $t_1=t_2=1$. We add two more parameters $t_1$ and $t_2$ in order to calculate the asymptotic formulas for plane partitions,  
without them Theorem \ref{th:Asy_Ribbon:multi} would not be proven.
Our further contribution is to reformulate the asymptotic formula in a 
much more simpler form \eqref{eq:asym_a}, \eqref{eq:asym_c} and \eqref{eq:asy_d_n}, so that the result of Theorem \ref{th:asy_main2} can be easily iterated 
for handling a product of multiple power series $F_1(q)F_2(q)\cdots F_k(q)$ (see Theorem \ref{th:asy_multi}). We also obtain the following two theorems, which are useful to find asymptotic formulas for various plane partitions. 

\begin{thm}\label{th:Asy_Ribbon:multi}
Let $m$ be a positive integer. 
Suppose that $x_i$ and 
	$y_i$ $(1\leq i\leq m)$ are 
positive integers such that 
$\gcd(x_1, x_2, \ldots, x_m, y_1, y_2, \ldots, y_m)=1$.
Then, the coefficients $d_n$ in the following infinite product
$$
\prod_{i=1}^m \prod_{k\geq 0}\frac {1}{1-q^{x_ik+y_i}}=\sum_{n=0}^\infty d_{n} q^{n}
$$
have the following asymptotic formula
\begin{align}\label{eq:asy_product_2_1}
d_{n} \ \sim  \ 
	 v {\frac{1}{2\sqrt{2\pi}}} 
	\frac{r^{b+1/4}}{n^{b+3/4}} \exp(\sqrt{nr}),
\end{align}
where
$$
	v=\prod_{i=1}^m \frac{\Gamma(y_i/x_i)}{\sqrt{x_i\pi}} (\frac {x_i}{2})^{y_i/x_i}, \qquad
	r=\sum_{i=1}^m\frac {2\pi^2}{3x_i}, \qquad 
	b=\sum_{i=1}^m(\frac{y_i}{2x_i} - \frac 14).
$$   
\end{thm}

\begin{thm}\label{th:p0_asy_main2}
Let $t_i\in \mathbb{N}$ for $1\leq i\leq m$.  Suppose that
\begin{equation*}
	F(q) = \sum_{n=0}^\infty a_{n} q^{n} \qquad\text{and}\qquad
F(q) \prod_{i=1}^m \frac{1}{1-q^{t_i}} = \sum_{n=0}^\infty d_n q^{n}
\end{equation*}
are two power series.
If	
\begin{align*}
a_{n} &\sim n^\alpha \exp(\beta n^p)
\end{align*}
where $0< p<1$, $\alpha\in \mathbb{R}, \beta>0$, then  we have
\begin{align}\label{eq:p0_asy_d_n}
d_n \sim \frac{n^{\alpha+m(1-p)}}{\beta^mp^m\prod_{i=1}^mt_i} \exp(\beta n^p).
\end{align}
\end{thm}

\medskip

An {\it ordinary plane partition} (\PP{PP}) is a filling $\omega=(\omega_{i,j})$ of
the quarter plane $\Lambda=\{(i,j)\mid i,j\geq 1\}$ 
 with
nonnegative integers
such that rows and columns decrease weakly, and 
the size $|\omega|=\sum \omega_{i,j}$ is finite.
The generating function for ordinary plane partitions  is known since MacMahon \cite{MacMahon1899, MacMahon1916, Stanley1971}: 
\begin{equation} \label{eq:gfPP}
	\sum_{\omega \in \PP{PP}} z^{|\omega|} 
	= 
	\prod_{i=1}^{\infty}\prod_{j=1}^{\infty}\frac{1}{1-z^{i+j-1}}=\prod_{i=1}^\infty (1-z^k)^{-k}. 
\end{equation}

For two partitions $\la$ and $\mu$,
	We write $\la \succ \mu$ or $\mu \prec \la$ if $\la / \mu$ is a horizontal strip (see \cite{Macdonald1995,  StanleyEC2}).
When reading an ordinary plane partition $\omega$ along the diagonals from left to right, we obtain a sequence of partitions
$(\la^0, \la^1, \ldots, \la^h)$
such that $\la^{i-1}\prec \la^i $ or $\la^{i-1} \succ \la^i$ for $1\leq i \leq h$. For simplicity, we identify the ordinary plane partition $\omega$ and the sequence of partitions by writing
$$
\omega=
(\la^0, \la^1, \ldots, \la^h).$$

A $(1,-1)$-sequence $\delta$ is called a $profile$. Let $|\delta|_1$
(resp. $|\delta|_{-1}$) be the number of letters $1$ (resp.  $-1$) in $\delta$.
A {\it skew plane partition} (\PP{SkewPP}) with profile $\delta=(\delta_1, \delta_2, \ldots, \delta_h)$ is a sequence of partitions $
\omega=
(\la^0, \la^1, \ldots, \la^h)$ 
such that $\la^0=\la^h=\emptyset$, and $\la^{i-1}\prec \la^i $ (resp. $\la^{i-1} \succ \la^i$)
if $\delta_i = 1$  (resp.  $\delta_i = -1$). Its  {\it size} is defined by $|\omega|=\sum_{i=0}^{h}|\la^i|$. 
For example,  $\omega=(\emptyset, (2),(3,2),(2),(3),(4,3),(3,2),(3),\emptyset)$ is a skew plane partition with profile $\delta=(1,1,-1,1,1,-1,-1,-1)$  and size $ 27 $. This skew plane partition can also be visualized as the following: 
$$
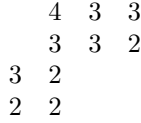
\begin{array}{llll}
\ & 4 & 3 & 3\\
\ & 3 & 3 & 2 \\
3 & 2 & &\\
{2} & 2 & &
\end{array}
$$

The generating function for skew plane partitions with profile $\delta$ is (see \cite{Borodin2007,  Sagan1993, Stanley1972})
\begin{equation}\label{eq:gfSkewPP}
\sum_{\omega \in \PP{SkewPP}_\delta} z^{|\omega|}
=
\prod_{\substack{ i<j   \\ \delta_i > \delta_j   } } 
\frac{1}{1-z^{j-i}}.
\end{equation}

\medskip

$$
\begin{tikzpicture}[scale=0.6]
\begin{scope}[xshift=0, yshift=0]
\fill [gray!30](0.0000,0.2000)--(0.0000,0.4000)--(0.0000,0.6000)--(0.0000,0.8000)--(0.0000,1.0000)--(0.0000,1.2000)--(0.0000,1.4000)--(0.0000,1.6000)--(0.0000,1.8000)--(0.0000,2.0000)--(0.0000,2.2000)--(0.0000,2.4000)--(0.0000,2.6000)--(0.0000,2.8000)--(0.0000,3.0000)--(0.0000,3.2000)--(0.0000,3.4000)--(0.2000,3.4000)--(0.4000,3.4000)--(0.6000,3.4000)--(0.8000,3.4000)--(1.0000,3.4000)--(1.2000,3.4000)--(1.4000,3.4000)--(1.6000,3.4000)--(1.8000,3.4000)--(2.0000,3.4000)--(2.2000,3.4000)--(2.4000,3.4000)--(2.6000,3.4000)--(2.8000,3.4000)--(3.0000,3.4000)--(3.2000,3.4000)--(3.4000,3.4000)--(3.6000,3.4000)--(3.8000,3.4000)--(4.0000,3.4000)--(4.0000,3.2000)--(4.0000,3.0000)--(4.0000,2.8000)--(4.0000,2.6000)--(4.0000,2.4000)--(4.0000,2.2000)--(4.0000,2.0000)--(4.0000,1.8000)--(4.0000,1.6000)--(4.0000,1.4000)--(4.0000,1.2000)--(4.0000,1.0000)--(4.0000,0.8000)--(4.0000,0.6000)--(4.0000,0.4000)--(4.0000,0.2000)--(3.8000,0.2000)--(3.6000,0.2000)--(3.4000,0.2000)--(3.2000,0.2000)--(3.0000,0.2000)--(2.8000,0.2000)--(2.6000,0.2000)--(2.4000,0.2000)--(2.2000,0.2000)--(2.0000,0.2000)--(1.8000,0.2000)--(1.6000,0.2000)--(1.4000,0.2000)--(1.2000,0.2000)--(1.0000,0.2000)--(0.8000,0.2000)--(0.6000,0.2000)--(0.4000,0.2000)--(0.2000,0.2000)--(0.0000,0.2000);
\draw [gray!10](0.0000,0.2000)--(0.0000,3.4000);
\draw [gray!10](0.2000,0.2000)--(0.2000,3.4000);
\draw [gray!10](0.4000,0.2000)--(0.4000,3.4000);
\draw [gray!10](0.6000,0.2000)--(0.6000,3.4000);
\draw [gray!10](0.8000,0.2000)--(0.8000,3.4000);
\draw [gray!10](1.0000,0.2000)--(1.0000,3.4000);
\draw [gray!10](1.2000,0.2000)--(1.2000,3.4000);
\draw [gray!10](1.4000,0.2000)--(1.4000,3.4000);
\draw [gray!10](1.6000,0.2000)--(1.6000,3.4000);
\draw [gray!10](1.8000,0.2000)--(1.8000,3.4000);
\draw [gray!10](2.0000,0.2000)--(2.0000,3.4000);
\draw [gray!10](2.2000,0.2000)--(2.2000,3.4000);
\draw [gray!10](2.4000,0.2000)--(2.4000,3.4000);
\draw [gray!10](2.6000,0.2000)--(2.6000,3.4000);
\draw [gray!10](2.8000,0.2000)--(2.8000,3.4000);
\draw [gray!10](3.0000,0.2000)--(3.0000,3.4000);
\draw [gray!10](3.2000,0.2000)--(3.2000,3.4000);
\draw [gray!10](3.4000,0.2000)--(3.4000,3.4000);
\draw [gray!10](3.6000,0.2000)--(3.6000,3.4000);
\draw [gray!10](3.8000,0.2000)--(3.8000,3.4000);
\draw [gray!10](0.0000,0.2000)--(4.0000,0.2000);
\draw [gray!10](0.0000,0.4000)--(4.0000,0.4000);
\draw [gray!10](0.0000,0.6000)--(4.0000,0.6000);
\draw [gray!10](0.0000,0.8000)--(4.0000,0.8000);
\draw [gray!10](0.0000,1.0000)--(4.0000,1.0000);
\draw [gray!10](0.0000,1.2000)--(4.0000,1.2000);
\draw [gray!10](0.0000,1.4000)--(4.0000,1.4000);
\draw [gray!10](0.0000,1.6000)--(4.0000,1.6000);
\draw [gray!10](0.0000,1.8000)--(4.0000,1.8000);
\draw [gray!10](0.0000,2.0000)--(4.0000,2.0000);
\draw [gray!10](0.0000,2.2000)--(4.0000,2.2000);
\draw [gray!10](0.0000,2.4000)--(4.0000,2.4000);
\draw [gray!10](0.0000,2.6000)--(4.0000,2.6000);
\draw [gray!10](0.0000,2.8000)--(4.0000,2.8000);
\draw [gray!10](0.0000,3.0000)--(4.0000,3.0000);
\draw [gray!10](0.0000,3.2000)--(4.0000,3.2000);
\draw [black](0.0000,0.2000)--(0.0000,0.4000)--(0.0000,0.6000)--(0.0000,0.8000)--(0.0000,1.0000)--(0.0000,1.2000)--(0.0000,1.4000)--(0.0000,1.6000)--(0.0000,1.8000)--(0.0000,2.0000)--(0.0000,2.2000)--(0.0000,2.4000)--(0.0000,2.6000)--(0.0000,2.8000)--(0.0000,3.0000)--(0.0000,3.2000)--(0.0000,3.4000)--(0.2000,3.4000)--(0.4000,3.4000)--(0.6000,3.4000)--(0.8000,3.4000)--(1.0000,3.4000)--(1.2000,3.4000)--(1.4000,3.4000)--(1.6000,3.4000)--(1.8000,3.4000)--(2.0000,3.4000)--(2.2000,3.4000)--(2.4000,3.4000)--(2.6000,3.4000)--(2.8000,3.4000)--(3.0000,3.4000)--(3.2000,3.4000)--(3.4000,3.4000)--(3.6000,3.4000)--(3.8000,3.4000)--(4.0000,3.4000);
\draw (0.0000,3.4000)--(0.0000,-0.2000);
\draw (0.0000,3.4000)--(4.4000,3.4000);
\draw (2,-1) node [] {A: Ordinary PP};
\end{scope}
\begin{scope}[xshift=160, yshift=0]
\fill [gray!30](2.6000,1.8000)--(2.4000,1.8000)--(2.2000,1.8000)--(2.0000,1.8000)--(1.8000,1.8000)--(1.6000,1.8000)--(1.4000,1.8000)--(1.2000,1.8000)--(1.0000,1.8000)--(0.8000,1.8000)--(0.6000,1.8000)--(0.4000,1.8000)--(0.2000,1.8000)--(0.0000,1.8000)--(0.0000,2.0000)--(0.0000,2.2000)--(0.0000,2.4000)--(0.0000,2.6000)--(0.0000,2.8000)--(0.2000,2.8000)--(0.4000,2.8000)--(0.4000,3.0000)--(0.6000,3.0000)--(0.8000,3.0000)--(1.0000,3.0000)--(1.0000,3.2000)--(1.2000,3.2000)--(1.2000,3.4000)--(1.4000,3.4000)--(1.6000,3.4000)--(1.8000,3.4000)--(2.0000,3.4000)--(2.2000,3.4000)--(2.4000,3.4000)--(2.6000,3.4000)--(2.6000,3.2000)--(2.6000,3.0000)--(2.6000,2.8000)--(2.6000,2.6000)--(2.6000,2.4000)--(2.6000,2.2000)--(2.6000,2.0000)--(2.6000,1.8000);
\draw [gray!10](0.0000,1.8000)--(0.0000,2.8000);
\draw [gray!10](0.2000,1.8000)--(0.2000,2.8000);
\draw [gray!10](0.4000,1.8000)--(0.4000,3.0000);
\draw [gray!10](0.6000,1.8000)--(0.6000,3.0000);
\draw [gray!10](0.8000,1.8000)--(0.8000,3.0000);
\draw [gray!10](1.0000,1.8000)--(1.0000,3.2000);
\draw [gray!10](1.2000,1.8000)--(1.2000,3.4000);
\draw [gray!10](1.4000,1.8000)--(1.4000,3.4000);
\draw [gray!10](1.6000,1.8000)--(1.6000,3.4000);
\draw [gray!10](1.8000,1.8000)--(1.8000,3.4000);
\draw [gray!10](2.0000,1.8000)--(2.0000,3.4000);
\draw [gray!10](2.2000,1.8000)--(2.2000,3.4000);
\draw [gray!10](2.4000,1.8000)--(2.4000,3.4000);
\draw [gray!10](0.0000,1.8000)--(2.6000,1.8000);
\draw [gray!10](0.0000,2.0000)--(2.6000,2.0000);
\draw [gray!10](0.0000,2.2000)--(2.6000,2.2000);
\draw [gray!10](0.0000,2.4000)--(2.6000,2.4000);
\draw [gray!10](0.0000,2.6000)--(2.6000,2.6000);
\draw [gray!10](0.0000,2.8000)--(2.6000,2.8000);
\draw [gray!10](0.4000,3.0000)--(2.6000,3.0000);
\draw [gray!10](1.0000,3.2000)--(2.6000,3.2000);
\draw [black](2.6000,1.8000)--(2.4000,1.8000)--(2.2000,1.8000)--(2.0000,1.8000)--(1.8000,1.8000)--(1.6000,1.8000)--(1.4000,1.8000)--(1.2000,1.8000)--(1.0000,1.8000)--(0.8000,1.8000)--(0.6000,1.8000)--(0.4000,1.8000)--(0.2000,1.8000)--(0.0000,1.8000)--(0.0000,2.0000)--(0.0000,2.2000)--(0.0000,2.4000)--(0.0000,2.6000)--(0.0000,2.8000)--(0.2000,2.8000)--(0.4000,2.8000)--(0.4000,3.0000)--(0.6000,3.0000)--(0.8000,3.0000)--(1.0000,3.0000)--(1.0000,3.2000)--(1.2000,3.2000)--(1.2000,3.4000)--(1.4000,3.4000)--(1.6000,3.4000)--(1.8000,3.4000)--(2.0000,3.4000)--(2.2000,3.4000)--(2.4000,3.4000)--(2.6000,3.4000)--(2.6000,3.2000)--(2.6000,3.0000)--(2.6000,2.8000)--(2.6000,2.6000)--(2.6000,2.4000)--(2.6000,2.2000)--(2.6000,2.0000)--(2.6000,1.8000);
\draw (0.0000,3.4000)--(0.0000,-0.2000);
\draw (0.0000,3.4000)--(4.2000,3.4000);
\draw (2,-1) node [] {B: Skew PP};
\end{scope}
\begin{scope}[xshift=320, yshift=0]
\fill [gray!30](1.6000,0.0000)--(1.4000,0.0000)--(1.4000,0.2000)--(1.2000,0.2000)--(1.2000,0.4000)--(1.0000,0.4000)--(1.0000,0.6000)--(0.8000,0.6000)--(0.8000,0.8000)--(0.6000,0.8000)--(0.6000,1.0000)--(0.4000,1.0000)--(0.4000,1.2000)--(0.2000,1.2000)--(0.2000,1.4000)--(0.0000,1.4000)--(0.0000,1.6000)--(0.0000,1.8000)--(0.0000,2.0000)--(0.0000,2.2000)--(0.0000,2.4000)--(0.2000,2.4000)--(0.2000,2.6000)--(0.4000,2.6000)--(0.6000,2.6000)--(0.6000,2.8000)--(0.8000,2.8000)--(0.8000,3.0000)--(0.8000,3.2000)--(1.0000,3.2000)--(1.0000,3.4000)--(1.2000,3.4000)--(1.4000,3.4000)--(1.6000,3.4000)--(1.8000,3.4000)--(2.0000,3.4000)--(2.0000,3.2000)--(2.2000,3.2000)--(2.2000,3.0000)--(2.4000,3.0000)--(2.4000,2.8000)--(2.6000,2.8000)--(2.6000,2.6000)--(2.8000,2.6000)--(2.8000,2.4000)--(3.0000,2.4000)--(3.0000,2.2000)--(3.2000,2.2000)--(3.2000,2.0000)--(3.4000,2.0000)--(3.4000,1.8000)--(3.2000,1.8000)--(3.2000,1.6000)--(3.0000,1.6000)--(3.0000,1.4000)--(2.8000,1.4000)--(2.8000,1.2000)--(2.6000,1.2000)--(2.6000,1.0000)--(2.4000,1.0000)--(2.4000,0.8000)--(2.2000,0.8000)--(2.2000,0.6000)--(2.0000,0.6000)--(2.0000,0.4000)--(1.8000,0.4000)--(1.8000,0.2000)--(1.6000,0.2000)--(1.6000,0.0000);
\draw [gray!10](0.0000,1.4000)--(0.0000,2.4000);
\draw [gray!10](0.2000,1.2000)--(0.2000,2.6000);
\draw [gray!10](0.4000,1.0000)--(0.4000,2.6000);
\draw [gray!10](0.6000,0.8000)--(0.6000,2.8000);
\draw [gray!10](0.8000,0.6000)--(0.8000,3.2000);
\draw [gray!10](1.0000,0.4000)--(1.0000,3.4000);
\draw [gray!10](1.2000,0.2000)--(1.2000,3.4000);
\draw [gray!10](1.4000,0.0000)--(1.4000,3.4000);
\draw [gray!10](1.6000,0.0000)--(1.6000,3.4000);
\draw [gray!10](1.8000,0.2000)--(1.8000,3.4000);
\draw [gray!10](2.0000,0.4000)--(2.0000,3.4000);
\draw [gray!10](2.2000,0.6000)--(2.2000,3.2000);
\draw [gray!10](2.4000,0.8000)--(2.4000,3.0000);
\draw [gray!10](2.6000,1.0000)--(2.6000,2.8000);
\draw [gray!10](2.8000,1.2000)--(2.8000,2.6000);
\draw [gray!10](3.0000,1.4000)--(3.0000,2.4000);
\draw [gray!10](3.2000,1.6000)--(3.2000,2.2000);
\draw [gray!10](1.4000,0.0000)--(1.6000,0.0000);
\draw [gray!10](1.2000,0.2000)--(1.8000,0.2000);
\draw [gray!10](1.0000,0.4000)--(2.0000,0.4000);
\draw [gray!10](0.8000,0.6000)--(2.2000,0.6000);
\draw [gray!10](0.6000,0.8000)--(2.4000,0.8000);
\draw [gray!10](0.4000,1.0000)--(2.6000,1.0000);
\draw [gray!10](0.2000,1.2000)--(2.8000,1.2000);
\draw [gray!10](0.0000,1.4000)--(3.0000,1.4000);
\draw [gray!10](0.0000,1.6000)--(3.2000,1.6000);
\draw [gray!10](0.0000,1.8000)--(3.4000,1.8000);
\draw [gray!10](0.0000,2.0000)--(3.4000,2.0000);
\draw [gray!10](0.0000,2.2000)--(3.2000,2.2000);
\draw [gray!10](0.0000,2.4000)--(3.0000,2.4000);
\draw [gray!10](0.2000,2.6000)--(2.8000,2.6000);
\draw [gray!10](0.6000,2.8000)--(2.6000,2.8000);
\draw [gray!10](0.8000,3.0000)--(2.4000,3.0000);
\draw [gray!10](0.8000,3.2000)--(2.2000,3.2000);
\draw [black](1.6000,0.0000)--(1.4000,0.0000)--(1.4000,0.2000)--(1.2000,0.2000)--(1.2000,0.4000)--(1.0000,0.4000)--(1.0000,0.6000)--(0.8000,0.6000)--(0.8000,0.8000)--(0.6000,0.8000)--(0.6000,1.0000)--(0.4000,1.0000)--(0.4000,1.2000)--(0.2000,1.2000)--(0.2000,1.4000)--(0.0000,1.4000)--(0.0000,1.6000)--(0.0000,1.8000)--(0.0000,2.0000)--(0.0000,2.2000)--(0.0000,2.4000)--(0.2000,2.4000)--(0.2000,2.6000)--(0.4000,2.6000)--(0.6000,2.6000)--(0.6000,2.8000)--(0.8000,2.8000)--(0.8000,3.0000)--(0.8000,3.2000)--(1.0000,3.2000)--(1.0000,3.4000)--(1.2000,3.4000)--(1.4000,3.4000)--(1.6000,3.4000)--(1.8000,3.4000)--(2.0000,3.4000)--(2.0000,3.2000)--(2.2000,3.2000)--(2.2000,3.0000)--(2.4000,3.0000)--(2.4000,2.8000)--(2.6000,2.8000)--(2.6000,2.6000)--(2.8000,2.6000)--(2.8000,2.4000)--(3.0000,2.4000)--(3.0000,2.2000)--(3.2000,2.2000)--(3.2000,2.0000)--(3.4000,2.0000)--(3.4000,1.8000);
\draw (0.0000,3.4000)--(0.0000,-0.2000);
\draw (0.0000,3.4000)--(3.6000,3.4000);
\draw (2.3,  2.4) node [] {$\lambda$};
\draw (1.1,  1.2) node [] {$\lambda$};
\draw (2,-1) node [] {C: Cylindric PP};
\end{scope}
\end{tikzpicture}
$$
\centerline{Fig. 1. Skew plane partitions and cylindric partitions.}
\medskip

{\it Cylindric partitions} (\PP{CP}) were first introduced by Gessel and Krattenthaler~\cite{GesselKratt1997}, see also \cite{Borodin2007} for an equivalent definition.
A cylindric partition 
with profile $\delta=(\delta_1, \delta_2, \ldots, \delta_h)$ is a sequence of partitions $
\omega=
(\la^0, \la^1, \ldots, \la^h)$ 
such that $\la^0=\la^h$, and  $\la^{i-1}\prec \la^i $ (resp. $\la^{i-1} \succ \la^i$)
if $\delta_i = 1$  (resp.  $\delta_i = -1$). Its  {\it size} is defined by $|\omega|=\sum_{i=0}^{h-1}|\la^i|$  (notice that $\la^h$ is not counted here, which is a little different from skew plane partitions).
For example, $\omega=( (2,1),(3,1),(4,1),(3),(4,2),(2,1) )$  is a cylindric partition with profile $\delta=(1,1,-1,1,-1)$ and size $ 21 $. This cylindric partition can be visualized as the following: 
$$
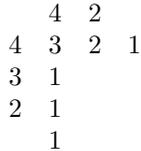
\begin{array}{llll}
\ & 4 & 2 & \\
4 & 3 & 2 & 1 \\
3 & 1 & &\\
{2} & 1 & &\\
\ & 1 & &
\end{array}
$$

Borodin obtained the generating function for cylindric partitions with profile 
$\delta=(\delta_i)_{1\leq i\leq h}$
(see \cite{Borodin2007,Langer2013B,Tingley2008}):
\begin{equation}\label{eq:gfCP}
	\sum_{\omega \in \PP{CP}_{\delta}} z^{|\omega|}
		=
		\prod_{k \geq 0} \left ( \frac{1}{1-z^{hk+h}} \prod_{\substack{i < j \\ \delta_i > \delta_j}} \frac{1}{1-z^{hk+j-i }} 
		\prod_{\substack{i < j \\ \delta_i < \delta_j }} \frac{1}{1 - z^{hk+h+i-j }} \right ). 
\end{equation}

\smallskip

By using Theorems \ref{th:Asy_Ribbon:multi} and \ref{th:p0_asy_main2}
we obtain 
the asymptotic formulas for the numbers of skew plane partitions
and 
cylindric partitions with size $n$ for fixed widths
in Sections~\ref{sec:PP} and~\ref{sec:CP} respectively.
Let us reproduce the asymptotic formulas for some special cases below:

\begin{align*}
\begin{tikzpicture}[scale=0.6]
\begin{scope}[xshift=0, yshift=0]
\fill [gray!30](1.8000,0.2000)--(1.6000,0.2000)--(1.4000,0.2000)--(1.2000,0.2000)--(1.0000,0.2000)--(0.8000,0.2000)--(0.6000,0.2000)--(0.4000,0.2000)--(0.2000,0.2000)--(0.0000,0.2000)--(-0.2000,0.2000)--(-0.4000,0.2000)--(-0.6000,0.2000)--(-0.8000,0.2000)--(-0.8000,0.4000)--(-0.8000,0.6000)--(-0.8000,0.8000)--(-0.6000,0.8000)--(-0.4000,0.8000)--(-0.2000,0.8000)--(0.0000,0.8000)--(0.2000,0.8000)--(0.4000,0.8000)--(0.6000,0.8000)--(0.8000,0.8000)--(1.0000,0.8000)--(1.2000,0.8000)--(1.4000,0.8000)--(1.6000,0.8000)--(1.8000,0.8000)--(1.8000,0.6000)--(1.8000,0.4000)--(1.8000,0.2000);
\draw [gray!10](-0.8000,0.2000)--(-0.8000,0.8000);
\draw [gray!10](-0.6000,0.2000)--(-0.6000,0.8000);
\draw [gray!10](-0.4000,0.2000)--(-0.4000,0.8000);
\draw [gray!10](-0.2000,0.2000)--(-0.2000,0.8000);
\draw [gray!10](0.0000,0.2000)--(0.0000,0.8000);
\draw [gray!10](0.2000,0.2000)--(0.2000,0.8000);
\draw [gray!10](0.4000,0.2000)--(0.4000,0.8000);
\draw [gray!10](0.6000,0.2000)--(0.6000,0.8000);
\draw [gray!10](0.8000,0.2000)--(0.8000,0.8000);
\draw [gray!10](1.0000,0.2000)--(1.0000,0.8000);
\draw [gray!10](1.2000,0.2000)--(1.2000,0.8000);
\draw [gray!10](1.4000,0.2000)--(1.4000,0.8000);
\draw [gray!10](1.6000,0.2000)--(1.6000,0.8000);
\draw [gray!10](-0.8000,0.2000)--(1.8000,0.2000);
\draw [gray!10](-0.8000,0.4000)--(1.8000,0.4000);
\draw [gray!10](-0.8000,0.6000)--(1.8000,0.6000);
\draw [black](1.8000,0.2000)--(1.6000,0.2000)--(1.4000,0.2000)--(1.2000,0.2000)--(1.0000,0.2000)--(0.8000,0.2000)--(0.6000,0.2000)--(0.4000,0.2000)--(0.2000,0.2000)--(0.0000,0.2000)--(-0.2000,0.2000)--(-0.4000,0.2000)--(-0.6000,0.2000)--(-0.8000,0.2000)--(-0.8000,0.4000)--(-0.8000,0.6000)--(-0.8000,0.8000)--(-0.6000,0.8000)--(-0.4000,0.8000)--(-0.2000,0.8000)--(0.0000,0.8000)--(0.2000,0.8000)--(0.4000,0.8000)--(0.6000,0.8000)--(0.8000,0.8000)--(1.0000,0.8000)--(1.2000,0.8000)--(1.4000,0.8000)--(1.6000,0.8000)--(1.8000,0.8000);
\draw (-0.8000,0.8000)--(-0.8000,-0.2000);
\draw (-0.8000,0.8000)--(2.2000,0.8000);
\draw (0.8,-1) node [] {PPa $\sim \frac{1.93}{n^3} e^{4.44 \sqrt{n}}$};
\end{scope}
\begin{scope}[xshift=180, yshift=0]
\fill [gray!30](1.8000,0.2000)--(1.6000,0.2000)--(1.4000,0.2000)--(1.2000,0.2000)--(1.0000,0.2000)--(0.8000,0.2000)--(0.6000,0.2000)--(0.4000,0.2000)--(0.2000,0.2000)--(0.0000,0.2000)--(-0.2000,0.2000)--(-0.4000,0.2000)--(-0.6000,0.2000)--(-0.8000,0.2000)--(-0.8000,0.4000)--(-0.8000,0.6000)--(-0.6000,0.6000)--(-0.6000,0.8000)--(-0.4000,0.8000)--(-0.2000,0.8000)--(0.0000,0.8000)--(0.2000,0.8000)--(0.4000,0.8000)--(0.6000,0.8000)--(0.8000,0.8000)--(1.0000,0.8000)--(1.2000,0.8000)--(1.4000,0.8000)--(1.6000,0.8000)--(1.8000,0.8000)--(1.8000,0.6000)--(1.8000,0.4000)--(1.8000,0.2000);
\draw [gray!10](-0.8000,0.2000)--(-0.8000,0.6000);
\draw [gray!10](-0.6000,0.2000)--(-0.6000,0.8000);
\draw [gray!10](-0.4000,0.2000)--(-0.4000,0.8000);
\draw [gray!10](-0.2000,0.2000)--(-0.2000,0.8000);
\draw [gray!10](0.0000,0.2000)--(0.0000,0.8000);
\draw [gray!10](0.2000,0.2000)--(0.2000,0.8000);
\draw [gray!10](0.4000,0.2000)--(0.4000,0.8000);
\draw [gray!10](0.6000,0.2000)--(0.6000,0.8000);
\draw [gray!10](0.8000,0.2000)--(0.8000,0.8000);
\draw [gray!10](1.0000,0.2000)--(1.0000,0.8000);
\draw [gray!10](1.2000,0.2000)--(1.2000,0.8000);
\draw [gray!10](1.4000,0.2000)--(1.4000,0.8000);
\draw [gray!10](1.6000,0.2000)--(1.6000,0.8000);
\draw [gray!10](-0.8000,0.2000)--(1.8000,0.2000);
\draw [gray!10](-0.8000,0.4000)--(1.8000,0.4000);
\draw [gray!10](-0.8000,0.6000)--(1.8000,0.6000);
\draw [black](1.8000,0.2000)--(1.6000,0.2000)--(1.4000,0.2000)--(1.2000,0.2000)--(1.0000,0.2000)--(0.8000,0.2000)--(0.6000,0.2000)--(0.4000,0.2000)--(0.2000,0.2000)--(0.0000,0.2000)--(-0.2000,0.2000)--(-0.4000,0.2000)--(-0.6000,0.2000)--(-0.8000,0.2000)--(-0.8000,0.4000)--(-0.8000,0.6000)--(-0.6000,0.6000)--(-0.6000,0.8000)--(-0.4000,0.8000)--(-0.2000,0.8000)--(0.0000,0.8000)--(0.2000,0.8000)--(0.4000,0.8000)--(0.6000,0.8000)--(0.8000,0.8000)--(1.0000,0.8000)--(1.2000,0.8000)--(1.4000,0.8000)--(1.6000,0.8000)--(1.8000,0.8000);
\draw (-0.8000,0.8000)--(-0.8000,-0.2000);
\draw (-0.8000,0.8000)--(2.2000,0.8000);
\draw (0.8,-1) node [] {PPb $\sim \frac{5.81}{n^3} e^{4.44 \sqrt{n}}$};
\end{scope}
\begin{scope}[xshift=360, yshift=0]
\fill [gray!30](1.8000,0.2000)--(1.6000,0.2000)--(1.4000,0.2000)--(1.2000,0.2000)--(1.0000,0.2000)--(0.8000,0.2000)--(0.6000,0.2000)--(0.4000,0.2000)--(0.2000,0.2000)--(0.0000,0.2000)--(-0.2000,0.2000)--(-0.4000,0.2000)--(-0.6000,0.2000)--(-0.8000,0.2000)--(-0.8000,0.4000)--(-0.8000,0.6000)--(-0.6000,0.6000)--(-0.4000,0.6000)--(-0.4000,0.8000)--(-0.2000,0.8000)--(0.0000,0.8000)--(0.2000,0.8000)--(0.4000,0.8000)--(0.6000,0.8000)--(0.8000,0.8000)--(1.0000,0.8000)--(1.2000,0.8000)--(1.4000,0.8000)--(1.6000,0.8000)--(1.8000,0.8000)--(1.8000,0.6000)--(1.8000,0.4000)--(1.8000,0.2000);
\draw [gray!10](-0.8000,0.2000)--(-0.8000,0.6000);
\draw [gray!10](-0.6000,0.2000)--(-0.6000,0.6000);
\draw [gray!10](-0.4000,0.2000)--(-0.4000,0.8000);
\draw [gray!10](-0.2000,0.2000)--(-0.2000,0.8000);
\draw [gray!10](0.0000,0.2000)--(0.0000,0.8000);
\draw [gray!10](0.2000,0.2000)--(0.2000,0.8000);
\draw [gray!10](0.4000,0.2000)--(0.4000,0.8000);
\draw [gray!10](0.6000,0.2000)--(0.6000,0.8000);
\draw [gray!10](0.8000,0.2000)--(0.8000,0.8000);
\draw [gray!10](1.0000,0.2000)--(1.0000,0.8000);
\draw [gray!10](1.2000,0.2000)--(1.2000,0.8000);
\draw [gray!10](1.4000,0.2000)--(1.4000,0.8000);
\draw [gray!10](1.6000,0.2000)--(1.6000,0.8000);
\draw [gray!10](-0.8000,0.2000)--(1.8000,0.2000);
\draw [gray!10](-0.8000,0.4000)--(1.8000,0.4000);
\draw [gray!10](-0.8000,0.6000)--(1.8000,0.6000);
\draw [black](1.8000,0.2000)--(1.6000,0.2000)--(1.4000,0.2000)--(1.2000,0.2000)--(1.0000,0.2000)--(0.8000,0.2000)--(0.6000,0.2000)--(0.4000,0.2000)--(0.2000,0.2000)--(0.0000,0.2000)--(-0.2000,0.2000)--(-0.4000,0.2000)--(-0.6000,0.2000)--(-0.8000,0.2000)--(-0.8000,0.4000)--(-0.8000,0.6000)--(-0.6000,0.6000)--(-0.4000,0.6000)--(-0.4000,0.8000)--(-0.2000,0.8000)--(0.0000,0.8000)--(0.2000,0.8000)--(0.4000,0.8000)--(0.6000,0.8000)--(0.8000,0.8000)--(1.0000,0.8000)--(1.2000,0.8000)--(1.4000,0.8000)--(1.6000,0.8000)--(1.8000,0.8000);
\draw (-0.8000,0.8000)--(-0.8000,-0.2000);
\draw (-0.8000,0.8000)--(2.2000,0.8000);
\draw (0.8,-1) node [] {PPc $\sim \frac{11.62}{n^3} e^{4.44 \sqrt{n}}$};
\end{scope}
\end{tikzpicture}
\\
\begin{tikzpicture}[scale=0.6]
\begin{scope}[xshift=0, yshift=0]
\fill [gray!30](3.2000,0.2000)--(3.2000,0.4000)--(3.0000,0.4000)--(3.0000,0.6000)--(2.8000,0.6000)--(2.8000,0.8000)--(2.6000,0.8000)--(2.6000,1.0000)--(2.4000,1.0000)--(2.4000,1.2000)--(2.2000,1.2000)--(2.2000,1.4000)--(2.0000,1.4000)--(2.0000,1.6000)--(1.8000,1.6000)--(1.8000,1.8000)--(1.6000,1.8000)--(1.6000,2.0000)--(1.4000,2.0000)--(1.4000,2.2000)--(1.2000,2.2000)--(1.2000,2.4000)--(1.0000,2.4000)--(1.0000,2.6000)--(0.8000,2.6000)--(0.8000,2.8000)--(0.8000,3.0000)--(1.0000,3.0000)--(1.2000,3.0000)--(1.4000,3.0000)--(1.6000,3.0000)--(1.6000,2.8000)--(1.8000,2.8000)--(1.8000,2.6000)--(2.0000,2.6000)--(2.0000,2.4000)--(2.2000,2.4000)--(2.2000,2.2000)--(2.4000,2.2000)--(2.4000,2.0000)--(2.6000,2.0000)--(2.6000,1.8000)--(2.8000,1.8000)--(2.8000,1.6000)--(3.0000,1.6000)--(3.0000,1.4000)--(3.2000,1.4000)--(3.2000,1.2000)--(3.4000,1.2000)--(3.4000,1.0000)--(3.6000,1.0000)--(3.6000,0.8000)--(3.8000,0.8000)--(3.8000,0.6000)--(3.6000,0.6000)--(3.6000,0.4000)--(3.4000,0.4000)--(3.4000,0.2000)--(3.2000,0.2000);
\draw [gray!10](0.8000,2.6000)--(0.8000,3.0000);
\draw [gray!10](1.0000,2.4000)--(1.0000,3.0000);
\draw [gray!10](1.2000,2.2000)--(1.2000,3.0000);
\draw [gray!10](1.4000,2.0000)--(1.4000,3.0000);
\draw [gray!10](1.6000,1.8000)--(1.6000,3.0000);
\draw [gray!10](1.8000,1.6000)--(1.8000,2.8000);
\draw [gray!10](2.0000,1.4000)--(2.0000,2.6000);
\draw [gray!10](2.2000,1.2000)--(2.2000,2.4000);
\draw [gray!10](2.4000,1.0000)--(2.4000,2.2000);
\draw [gray!10](2.6000,0.8000)--(2.6000,2.0000);
\draw [gray!10](2.8000,0.6000)--(2.8000,1.8000);
\draw [gray!10](3.0000,0.4000)--(3.0000,1.6000);
\draw [gray!10](3.2000,0.2000)--(3.2000,1.4000);
\draw [gray!10](3.4000,0.2000)--(3.4000,1.2000);
\draw [gray!10](3.6000,0.4000)--(3.6000,1.0000);
\draw [gray!10](3.2000,0.2000)--(3.4000,0.2000);
\draw [gray!10](3.0000,0.4000)--(3.6000,0.4000);
\draw [gray!10](2.8000,0.6000)--(3.8000,0.6000);
\draw [gray!10](2.6000,0.8000)--(3.8000,0.8000);
\draw [gray!10](2.4000,1.0000)--(3.6000,1.0000);
\draw [gray!10](2.2000,1.2000)--(3.4000,1.2000);
\draw [gray!10](2.0000,1.4000)--(3.2000,1.4000);
\draw [gray!10](1.8000,1.6000)--(3.0000,1.6000);
\draw [gray!10](1.6000,1.8000)--(2.8000,1.8000);
\draw [gray!10](1.4000,2.0000)--(2.6000,2.0000);
\draw [gray!10](1.2000,2.2000)--(2.4000,2.2000);
\draw [gray!10](1.0000,2.4000)--(2.2000,2.4000);
\draw [gray!10](0.8000,2.6000)--(2.0000,2.6000);
\draw [gray!10](0.8000,2.8000)--(1.8000,2.8000);
\draw [black](3.2000,0.2000)--(3.2000,0.4000)--(3.0000,0.4000)--(3.0000,0.6000)--(2.8000,0.6000)--(2.8000,0.8000)--(2.6000,0.8000)--(2.6000,1.0000)--(2.4000,1.0000)--(2.4000,1.2000)--(2.2000,1.2000)--(2.2000,1.4000)--(2.0000,1.4000)--(2.0000,1.6000)--(1.8000,1.6000)--(1.8000,1.8000)--(1.6000,1.8000)--(1.6000,2.0000)--(1.4000,2.0000)--(1.4000,2.2000)--(1.2000,2.2000)--(1.2000,2.4000)--(1.0000,2.4000)--(1.0000,2.6000)--(0.8000,2.6000)--(0.8000,2.8000)--(0.8000,3.0000)--(1.0000,3.0000)--(1.2000,3.0000)--(1.4000,3.0000)--(1.6000,3.0000)--(1.6000,2.8000)--(1.8000,2.8000)--(1.8000,2.6000)--(2.0000,2.6000)--(2.0000,2.4000)--(2.2000,2.4000)--(2.2000,2.2000)--(2.4000,2.2000)--(2.4000,2.0000)--(2.6000,2.0000)--(2.6000,1.8000)--(2.8000,1.8000)--(2.8000,1.6000)--(3.0000,1.6000)--(3.0000,1.4000)--(3.2000,1.4000)--(3.2000,1.2000)--(3.4000,1.2000)--(3.4000,1.0000)--(3.6000,1.0000)--(3.6000,0.8000)--(3.8000,0.8000);
\draw (0.8000,3.0000)--(0.8000,-0.2000);
\draw (0.8000,3.0000)--(4.2000,3.0000);
\draw (3.2,  2.0) node [] {$\lambda$};
\draw (2.0,  0.8) node [] {$\lambda$};
\draw (2.5,-1) node [] {CPa $\sim \frac{0.144}n e^{2.56 \sqrt{n}}$};
\end{scope}
\begin{scope}[xshift=180, yshift=0]
\fill [gray!30](3.2000,0.2000)--(3.2000,0.4000)--(3.0000,0.4000)--(3.0000,0.6000)--(2.8000,0.6000)--(2.8000,0.8000)--(2.6000,0.8000)--(2.6000,1.0000)--(2.4000,1.0000)--(2.4000,1.2000)--(2.2000,1.2000)--(2.2000,1.4000)--(2.0000,1.4000)--(2.0000,1.6000)--(1.8000,1.6000)--(1.8000,1.8000)--(1.6000,1.8000)--(1.6000,2.0000)--(1.4000,2.0000)--(1.4000,2.2000)--(1.2000,2.2000)--(1.2000,2.4000)--(1.0000,2.4000)--(1.0000,2.6000)--(1.0000,2.8000)--(1.2000,2.8000)--(1.2000,3.0000)--(1.4000,3.0000)--(1.6000,3.0000)--(1.6000,2.8000)--(1.8000,2.8000)--(1.8000,2.6000)--(2.0000,2.6000)--(2.0000,2.4000)--(2.2000,2.4000)--(2.2000,2.2000)--(2.4000,2.2000)--(2.4000,2.0000)--(2.6000,2.0000)--(2.6000,1.8000)--(2.8000,1.8000)--(2.8000,1.6000)--(3.0000,1.6000)--(3.0000,1.4000)--(3.2000,1.4000)--(3.2000,1.2000)--(3.4000,1.2000)--(3.4000,1.0000)--(3.6000,1.0000)--(3.6000,0.8000)--(3.8000,0.8000)--(3.8000,0.6000)--(3.6000,0.6000)--(3.6000,0.4000)--(3.4000,0.4000)--(3.4000,0.2000)--(3.2000,0.2000);
\draw [gray!10](1.0000,2.4000)--(1.0000,2.8000);
\draw [gray!10](1.2000,2.2000)--(1.2000,3.0000);
\draw [gray!10](1.4000,2.0000)--(1.4000,3.0000);
\draw [gray!10](1.6000,1.8000)--(1.6000,3.0000);
\draw [gray!10](1.8000,1.6000)--(1.8000,2.8000);
\draw [gray!10](2.0000,1.4000)--(2.0000,2.6000);
\draw [gray!10](2.2000,1.2000)--(2.2000,2.4000);
\draw [gray!10](2.4000,1.0000)--(2.4000,2.2000);
\draw [gray!10](2.6000,0.8000)--(2.6000,2.0000);
\draw [gray!10](2.8000,0.6000)--(2.8000,1.8000);
\draw [gray!10](3.0000,0.4000)--(3.0000,1.6000);
\draw [gray!10](3.2000,0.2000)--(3.2000,1.4000);
\draw [gray!10](3.4000,0.2000)--(3.4000,1.2000);
\draw [gray!10](3.6000,0.4000)--(3.6000,1.0000);
\draw [gray!10](3.2000,0.2000)--(3.4000,0.2000);
\draw [gray!10](3.0000,0.4000)--(3.6000,0.4000);
\draw [gray!10](2.8000,0.6000)--(3.8000,0.6000);
\draw [gray!10](2.6000,0.8000)--(3.8000,0.8000);
\draw [gray!10](2.4000,1.0000)--(3.6000,1.0000);
\draw [gray!10](2.2000,1.2000)--(3.4000,1.2000);
\draw [gray!10](2.0000,1.4000)--(3.2000,1.4000);
\draw [gray!10](1.8000,1.6000)--(3.0000,1.6000);
\draw [gray!10](1.6000,1.8000)--(2.8000,1.8000);
\draw [gray!10](1.4000,2.0000)--(2.6000,2.0000);
\draw [gray!10](1.2000,2.2000)--(2.4000,2.2000);
\draw [gray!10](1.0000,2.4000)--(2.2000,2.4000);
\draw [gray!10](1.0000,2.6000)--(2.0000,2.6000);
\draw [gray!10](1.0000,2.8000)--(1.8000,2.8000);
\draw [black](3.2000,0.2000)--(3.2000,0.4000)--(3.0000,0.4000)--(3.0000,0.6000)--(2.8000,0.6000)--(2.8000,0.8000)--(2.6000,0.8000)--(2.6000,1.0000)--(2.4000,1.0000)--(2.4000,1.2000)--(2.2000,1.2000)--(2.2000,1.4000)--(2.0000,1.4000)--(2.0000,1.6000)--(1.8000,1.6000)--(1.8000,1.8000)--(1.6000,1.8000)--(1.6000,2.0000)--(1.4000,2.0000)--(1.4000,2.2000)--(1.2000,2.2000)--(1.2000,2.4000)--(1.0000,2.4000)--(1.0000,2.6000)--(1.0000,2.8000)--(1.2000,2.8000)--(1.2000,3.0000)--(1.4000,3.0000)--(1.6000,3.0000)--(1.6000,2.8000)--(1.8000,2.8000)--(1.8000,2.6000)--(2.0000,2.6000)--(2.0000,2.4000)--(2.2000,2.4000)--(2.2000,2.2000)--(2.4000,2.2000)--(2.4000,2.0000)--(2.6000,2.0000)--(2.6000,1.8000)--(2.8000,1.8000)--(2.8000,1.6000)--(3.0000,1.6000)--(3.0000,1.4000)--(3.2000,1.4000)--(3.2000,1.2000)--(3.4000,1.2000)--(3.4000,1.0000)--(3.6000,1.0000)--(3.6000,0.8000)--(3.8000,0.8000);
\draw (1.0000,3.0000)--(1.0000,-0.2000);
\draw (1.0000,3.0000)--(4.2000,3.0000);
\draw (3.2,  2.0) node [] {$\lambda$};
\draw (2.0,  0.8) node [] {$\lambda$};
\draw (2.5,-1) node [] {CPb $\sim \frac{0.161}n e^{2.86 \sqrt{n}}$};
\end{scope}
\begin{scope}[xshift=360, yshift=0]
\fill [gray!30](3.2000,0.2000)--(3.2000,0.4000)--(3.0000,0.4000)--(3.0000,0.6000)--(2.8000,0.6000)--(2.8000,0.8000)--(2.6000,0.8000)--(2.6000,1.0000)--(2.4000,1.0000)--(2.4000,1.2000)--(2.2000,1.2000)--(2.2000,1.4000)--(2.0000,1.4000)--(2.0000,1.6000)--(1.8000,1.6000)--(1.8000,1.8000)--(1.6000,1.8000)--(1.6000,2.0000)--(1.4000,2.0000)--(1.4000,2.2000)--(1.2000,2.2000)--(1.2000,2.4000)--(1.0000,2.4000)--(1.0000,2.6000)--(1.0000,2.8000)--(1.2000,2.8000)--(1.4000,2.8000)--(1.4000,3.0000)--(1.6000,3.0000)--(1.6000,2.8000)--(1.8000,2.8000)--(1.8000,2.6000)--(2.0000,2.6000)--(2.0000,2.4000)--(2.2000,2.4000)--(2.2000,2.2000)--(2.4000,2.2000)--(2.4000,2.0000)--(2.6000,2.0000)--(2.6000,1.8000)--(2.8000,1.8000)--(2.8000,1.6000)--(3.0000,1.6000)--(3.0000,1.4000)--(3.2000,1.4000)--(3.2000,1.2000)--(3.4000,1.2000)--(3.4000,1.0000)--(3.6000,1.0000)--(3.6000,0.8000)--(3.8000,0.8000)--(3.8000,0.6000)--(3.6000,0.6000)--(3.6000,0.4000)--(3.4000,0.4000)--(3.4000,0.2000)--(3.2000,0.2000);
\draw [gray!10](1.0000,2.4000)--(1.0000,2.8000);
\draw [gray!10](1.2000,2.2000)--(1.2000,2.8000);
\draw [gray!10](1.4000,2.0000)--(1.4000,3.0000);
\draw [gray!10](1.6000,1.8000)--(1.6000,3.0000);
\draw [gray!10](1.8000,1.6000)--(1.8000,2.8000);
\draw [gray!10](2.0000,1.4000)--(2.0000,2.6000);
\draw [gray!10](2.2000,1.2000)--(2.2000,2.4000);
\draw [gray!10](2.4000,1.0000)--(2.4000,2.2000);
\draw [gray!10](2.6000,0.8000)--(2.6000,2.0000);
\draw [gray!10](2.8000,0.6000)--(2.8000,1.8000);
\draw [gray!10](3.0000,0.4000)--(3.0000,1.6000);
\draw [gray!10](3.2000,0.2000)--(3.2000,1.4000);
\draw [gray!10](3.4000,0.2000)--(3.4000,1.2000);
\draw [gray!10](3.6000,0.4000)--(3.6000,1.0000);
\draw [gray!10](3.2000,0.2000)--(3.4000,0.2000);
\draw [gray!10](3.0000,0.4000)--(3.6000,0.4000);
\draw [gray!10](2.8000,0.6000)--(3.8000,0.6000);
\draw [gray!10](2.6000,0.8000)--(3.8000,0.8000);
\draw [gray!10](2.4000,1.0000)--(3.6000,1.0000);
\draw [gray!10](2.2000,1.2000)--(3.4000,1.2000);
\draw [gray!10](2.0000,1.4000)--(3.2000,1.4000);
\draw [gray!10](1.8000,1.6000)--(3.0000,1.6000);
\draw [gray!10](1.6000,1.8000)--(2.8000,1.8000);
\draw [gray!10](1.4000,2.0000)--(2.6000,2.0000);
\draw [gray!10](1.2000,2.2000)--(2.4000,2.2000);
\draw [gray!10](1.0000,2.4000)--(2.2000,2.4000);
\draw [gray!10](1.0000,2.6000)--(2.0000,2.6000);
\draw [gray!10](1.0000,2.8000)--(1.8000,2.8000);
\draw [black](3.2000,0.2000)--(3.2000,0.4000)--(3.0000,0.4000)--(3.0000,0.6000)--(2.8000,0.6000)--(2.8000,0.8000)--(2.6000,0.8000)--(2.6000,1.0000)--(2.4000,1.0000)--(2.4000,1.2000)--(2.2000,1.2000)--(2.2000,1.4000)--(2.0000,1.4000)--(2.0000,1.6000)--(1.8000,1.6000)--(1.8000,1.8000)--(1.6000,1.8000)--(1.6000,2.0000)--(1.4000,2.0000)--(1.4000,2.2000)--(1.2000,2.2000)--(1.2000,2.4000)--(1.0000,2.4000)--(1.0000,2.6000)--(1.0000,2.8000)--(1.2000,2.8000)--(1.4000,2.8000)--(1.4000,3.0000)--(1.6000,3.0000)--(1.6000,2.8000)--(1.8000,2.8000)--(1.8000,2.6000)--(2.0000,2.6000)--(2.0000,2.4000)--(2.2000,2.4000)--(2.2000,2.2000)--(2.4000,2.2000)--(2.4000,2.0000)--(2.6000,2.0000)--(2.6000,1.8000)--(2.8000,1.8000)--(2.8000,1.6000)--(3.0000,1.6000)--(3.0000,1.4000)--(3.2000,1.4000)--(3.2000,1.2000)--(3.4000,1.2000)--(3.4000,1.0000)--(3.6000,1.0000)--(3.6000,0.8000)--(3.8000,0.8000);
\draw (1.0000,3.0000)--(1.0000,-0.2000);
\draw (1.0000,3.0000)--(4.2000,3.0000);
\draw (3.2,  2.0) node [] {$\lambda$};
\draw (2.0,  0.8) node [] {$\lambda$};
\draw (2.5,-1) node [] {CPc $\sim \frac{0.114}n e^{2.86 \sqrt{n}}$};
\end{scope}
\end{tikzpicture}
\end{align*}

\centerline{Fig. 2. Asymptotic formulas for skew PP and CP of fixed widths.}
\medskip

We see that the order of the asymptotic formula for skew plane partitions of fixed width
depends only on the width, 
not on the profile (the skew zone) itself. 
We may think that this is natural by intuition. 
However, the case for cylindric partitions shows that this is not always true.

\medskip
The rest of the paper is arranged in the following way. 
First, in Section~\ref{sec:proofmain} we prove our main theorems
on asymptotic formulas.
Later,
we compute the asymptotic formulas
for the numbers of skew plane partitions and cylindric partitions in Sections \ref{sec:PP} and \ref{sec:CP} respectively.


\section{Proofs of main asymptotic formulas}\label{sec:proofmain} 

In this section we prove the three main asymptotic formulas stated in Theorems~\ref{th:asy_main2}, 
\ref{th:Asy_Ribbon:multi} and \ref{th:p0_asy_main2}.
The basic idea of the proofs  comes from the work of Dewar and Murty \cite{DewarMurty2013}. First let us recall Laplace's method  (see, for example, \cite[p. 36]{Erdelyi1956}).

\begin{lem}[Laplace's method]\label{th:Laplace_method} 
Assume that $f (x) $  is a twice continuously differentiable function on $[a,b ]$ with $ x_{0}\in (a,b)$ the unique point such that  $ f(x_{0})=\max _{[a,b]}f(x)$. Assume additionally that $f''( x_0 ) < 0$. Then 
$$
\int_a^b {{e}^{  n f\left( x \right)} 
dx} \sim 
{e}^{  n f\left( {x_0} \right)} \sqrt {\frac{{2\pi }}{{- n f''\left( {x_0} \right)}}} .
$$
The sign $\sim$ means that the quotient of the left-hand side by the right-hand side tends to $1$ as $n \to +\infty$.  
\end{lem}

We also need the following lemma.

\begin{lem}\label{th:asy_int}
Suppose that $n$ is a positive integer.
Let $f (x) $ be a non-negative Lebesgue integrable function on $[a,b]$ 
with $ x_{0}\in (a,b)$ the unique point such that  $ f(x_{0})=\max _{[a,b]}f(x)$. 
Assume additionally that $f( x ) $ increases on $(a, x_0)$ and decreases on $(x_0, b)$. Then, 
$$
\int_a^b f(x) dx - \frac{f(x_0)}{n} \leq  \frac1n\sum_{i=\ceil*{na}}^{\floor*{nb}} f(\frac{i}{n}) \leq 
\int_a^b f(x) dx + \frac{f(x_0)}{n}.
$$

\end{lem}
\begin{proof}
Let $f(x)=0$ for $x\notin [a,b]$. 
Since $f( x ) $ increases on $(a, x_0)$, we have 
$$
\frac{f(k)}{n} \leq \int_{k}^{k+\frac1n} f(x) dx \leq \frac{f(k+\frac1n)}{n}
$$
when $a\leq k< k+\frac1n\leq x_0$. Since $f(x)$ decreases on $(x_0, b)$,
we obtain
$$
\frac{f(k+\frac1n)}{n} \leq \int_{k}^{k+\frac1n} f(x) dx \leq \frac{f(k)}{n}
$$
when $x_0\leq k< k+\frac1n\leq b$.

Let $k_0$ be the integer such that $\frac{k_0}{n}\leq x_0 < \frac{k_0+1}{n},$ we have
$$
\int_{\frac{k_0}{n}}^{\frac{k_0+1}{n}} f(x) dx \leq \frac{f(x_0)}{n}. 
$$
Therefore 
\begin{align*}
\int_a^b f(x) dx 
&\leq
\int_{\frac{\ceil*{na}-1}{n}}^{ \frac{\floor*{nb}+1}{n} } f(x) dx
\\&=
\int_{\frac{\ceil*{na}-1}{n}}^{ \frac{k_0}{n} } f(x) dx
+
\int_{\frac{k_0}{n}}^{ \frac{k_0+1}{n}} f(x) dx
+
\int_{\frac{k_0+1}{n}}^{ \frac{\floor*{nb}+1}{n} } f(x) dx
\\&\leq
\frac1n\sum_{i=\ceil*{na}}^{k_0} f(\frac{i}{n})
+
\frac1n f(x_0)
+
\frac1n\sum_{i=k_0+1}^{\floor*{nb}} f(\frac{i}{n})
\\&\leq
\frac1n\sum_{i=\ceil*{na}}^{\floor*{nb}} f(\frac{i}{n})
+
\frac1n f(x_0).
\end{align*}

On the other hand,
\begin{align*}
\int_a^b f(x) dx 
&\geq
\int_{\frac{\ceil*{na}}{n}}^{ \frac{\floor*{nb}}{n} } f(x) dx
\\&=
\int_{\frac{\ceil*{na}}{n}}^{ \frac{k_0}{n} } f(x) dx
+
\int_{\frac{k_0}{n}}^{ \frac{k_0+1}{n}} f(x) dx
+
\int_{\frac{k_0+1}{n}}^{ \frac{\floor*{nb}}{n} } f(x) dx
\\&\geq
\frac1n\sum_{i=\ceil*{na}}^{k_0-1} f(\frac{i}{n})
+
\frac1n 
\left(
-f(x_0)+f(\frac{k_0}{n})+f(\frac{k_0+1}{n})
\right)
+
\frac1n\sum_{i=k_0+2}^{\floor*{nb}} f(\frac{i}{n})
\\&\geq
\frac1n\sum_{i=\ceil*{na}}^{\floor*{nb}} f(\frac{i}{n})
-
\frac1n f(x_0).\qedhere
\end{align*}
\end{proof}

\medskip

Now we can give the proof of Theorem \ref{th:asy_main2}.
\begin{proof}[Proof of Theorem \ref{th:asy_main2}]
Without loss of generality, we can assume that $v_1,v_2>0.$
For $0<x<1$, let 
$$
f(x)=r_1^{1-p}x^p+r_2^{1-p}(1-x)^p 
$$
and
$$
g(x)=x^{-b_1-1+\frac{p}{2}}(1-x)^{-b_2-1+\frac{p}{2}}. 
$$
Then 
$$
f'(x)=pr_1^{1-p}x^{p-1}-pr_2^{1-p}(1-x)^{p-1}
$$
and 
$$
f''(x)=p(p-1)r_1^{1-p}x^{p-2}+p(p-1)r_2^{1-p}(1-x)^{p-2}<0.
$$
The function $f'(x)$ has only one zero point $$x_0=\frac{r_1}{r_1+r_2}.$$ Therefore $f(x)$ is increasing on $(0, x_0)$, has a maximum of
$(r_1+r_2)^{1-p}$ at~$x_0$, and is decreasing on $(x_0, 1)$.

\medskip

Let $0< \epsilon < 1$ be a given constant. By continuity, there exists $0<\delta<\min\{\frac12 x_0,\frac12 (1-x_0)\}$
such that if $|x - x_0| < 2\delta$, then
\begin{align}\label{eq:asy_proof_1}
(1 - \epsilon)g(x_0) < g(x) < (1 + \epsilon)g(x_0). 
\end{align}
From \eqref{eq:asym_a} and \eqref{eq:asym_c}, for large enough $n$ we have
\begin{align}\label{eq:asy_proof_2}
(1 - \epsilon)t_1\psi_{t_1n}(v_1, r_1, b_1 ; p)< a_{t_1n} < (1 + \epsilon)t_1\psi_{t_1n}(v_1, r_1, b_1; p)
\end{align}
and
\begin{align}\label{eq:asy_proof_3}
(1 - \epsilon)t_2\psi_{t_2n}(v_2, r_2, b_2; p)< c_{t_2n} < (1 + \epsilon)t_2\psi_{t_2n}(v_2, r_2, b_2; p).
\end{align}
\medskip

Suppose that $0\leq i \leq t_1t_2-1$ is a given integer. We just need to prove that \eqref{eq:asy_d_n} is true for $n=mt_1t_2+i$ where $m \in \mathbb{N}$.
By B\'ezout's identity, there exists some $\alpha_i,\beta_i\in\mathbb{N}_{\geq 0}, 0\leq \alpha_i\leq t_2-1$ such that
$$
t_1\alpha_i+t_2\beta_i=i.
$$

For large enough $n=mt_1t_2+i$, let
\begin{align*}
	j_1(n)&=\ceil*{\frac{(x_0-\delta)n-\alpha_it_1}{t_1t_2}}, \\
	j_2(n)&=\floor*{\frac{(x_0+\delta)n-\alpha_it_1}{t_1t_2} },\\
 	j_3(n)&=\floor*{\frac{n-\alpha_it_1}{t_1t_2}}.
\end{align*}
We have
$$
d_{n}=  H_1({n})+H_2({n})+H_3({n}),
$$
where 
\begin{align*}
	H_1({n})&= \sum_{ j=0 }^{j_1(n)-1} a_{(\alpha_i+jt_2)t_1}c_{(mt_1-jt_1+\beta_i)t_2},\\
	H_2({n})&= \sum_{ j=j_2(n)+1}^{j_3(n)} a_{(\alpha_i+jt_2)t_1}c_{(mt_1-jt_1+\beta_i)t_2},\\
	H_3({n})&= \sum_{ j=j_1(n) }^{j_2(n)} a_{(\alpha_i+jt_2)t_1}c_{(mt_1-jt_1+\beta_i)t_2}.
\end{align*}
For $H_1(n)$, we have 
\begin{align*}
H_1(n)
&=
	O\left(n^{|b_1+1-\frac{p}{2}|+|b_2+1-\frac{p}{2}|}  \sum_{ j=0 }^{j_1(n)-1} \exp\left(  n^p f\left(\frac{(\alpha_i+jt_2)t_1}{n}\right)   \right) \right)
\\&=
O\left(n^{|b_1+1-\frac{p}{2}|+|b_2+1-\frac{p}{2}|+1}   \exp\left(  n^p f(x_0-\delta)   \right) \right)
\\&=
o\left(n^{-b_1-b_2- 1+ \frac{p}{2}} \exp\left(  n^p f(x_0)   \right) \right).
\end{align*}
Similarly, we have 
\begin{align*}
H_2(n)
&=
o\left(n^{-b_1-b_2- 1+ \frac{p}{2}} \exp\left(  n^p f(x_0)   \right) \right).
\end{align*}

\medskip

Next we just need to estimate $H_3(n)$. For large enough $n$, we can 
assume that every 
$a_{(\alpha_i+jt_2)t_1}$ and $c_{(mt_1-jt_1+\beta_i)t_2}$
in $H_3(n)$ satisfy \eqref{eq:asy_proof_2} and \eqref{eq:asy_proof_3}.
Let
\begin{align*}
	A_0&=  {\frac{p(1-p)}{2\pi}}  t_1 t_2 v_1v_2 {r_1^{b_1+(1-p)/2}} {r_2^{b_2+(1-p)/2}},\\
	A_1(n)&=
A_0 n^{-b_1-b_2- 2+ {p}}\sum_{ j=j_1(n)} ^{j_2(n)} g(\frac{(\alpha_i+jt_2)t_1}{n})
\exp\left(  n^p f\left(\frac{(\alpha_i+jt_2)t_1}{n}\right)   \right),\\
	A_2(n)&=g(x_0) A_0 n^{-b_1-b_2- 2+ {p}}\sum_{ j=j_1(n)}^{j_2(n)} 
\exp\left(  n^p f\left(\frac{(\alpha_i+jt_2)t_1}{n}\right)   \right),\\
	A_3(n)&=
\int_{\frac{(x_0-\delta)n-\alpha_it_1}{t_1t_2n}}^{\frac{(x_0+\delta)n-\alpha_it_1}{t_1t_2n}} \exp\left(  n^p f(\frac{(\alpha_i+nxt_2)t_1}{n})    \right) dx.
\end{align*}

Therefore 
\begin{equation*}
	(1 - \epsilon)^2 A_1(n)< H_3(n) < (1 + \epsilon)^2 A_1(n).
\end{equation*}
Then by \eqref{eq:asy_proof_1}, we obtain
\begin{equation} \label{eq:asy_proof_4}
(1 - \epsilon)^3 A_2(n) < H_3(n)
	< (1 + \epsilon)^3 A_2(n).
\end{equation}

Replace $f(x)$ by $\exp\left(  n^p f(\frac{(\alpha_i+nxt_2)t_1}{n})   \right)$ in Lemma \ref{th:asy_int}, we have 
\begin{align} \label{eq:asy_proof_5}
	A_3(n)- \frac{\exp\left(  n^p f(x_0)   \right)}{n} 
&\leq
	\frac1n\sum_{j=j_1(n)}^{j_2(n)}   \exp\left(  n^p f\left(\frac{(\alpha_i+jt_2)t_1}{n}\right)   \right)    
\\&\leq 
	A_3(n)+ \frac{\exp\left(  n^p f(x_0)   \right)}{n}. \nonumber
\end{align}

Put \eqref{eq:asy_proof_4} and \eqref{eq:asy_proof_5} together, we obtain
\begin{align} \label{eq:asy_proof_6}
&
(1 - \epsilon)^3
\left(
	A_3(n) - \frac{\exp\left(  n^p f(x_0)   \right)}{n}
\right)
	< \frac{ H_3(n)}{ g(x_0)A_0n^{-b_1-b_2- 1+ {p}}}
	\\&
\qquad <
(1 + \epsilon)^3
\left(
	A_3(n) + \frac{\exp\left(  n^p f(x_0)   \right)}{n}
\right). \nonumber
\end{align}

Notice that when $n$ is large enough, 
\begin{align} \label{eq: x0-delta}
\frac{x_0-\frac{3\delta}{2}}{t_1t_2}<\frac{(x_0-\delta)n-\alpha_it_1}{t_1t_2n}<\frac{x_0-\frac{\delta}{2}}{t_1t_2},
\end{align}
\begin{align} \label{eq: x0+delta}
\frac{x_0+\frac{\delta}{2}}{t_1t_2}<\frac{(x_0+\delta)n-\alpha_it_1}{t_1t_2n}<\frac{x_0+\frac{3\delta}{2}}{t_1t_2}.
\end{align}
Also we have
\begin{align*}
 n^p f\left(\frac{(\alpha_i+nxt_2)t_1}{n}\right)   
	&=
r_1^{1-p}\left({(\alpha_i+nxt_2)t_1}\right)^p+ r_2^{1-p}\left(1- {(\alpha_i+nxt_2)t_1}\right)^p \\
	&= n^pf(xt_1t_2)+o(1).
\end{align*}

Then by \eqref{eq: x0-delta}, \eqref{eq: x0+delta} and Lemma \ref{th:Laplace_method} (\emph{Laplace's method}) we have 
\begin{align}
	A_3(n)
& \sim 
 \int_{\frac{(x_0-\delta)n-\alpha_it_1}{t_1t_2n}}^{\frac{(x_0+\delta)n-\alpha_it_1}{t_1t_2n}} \exp\left(  n^p f(xt_1t_2)   \right) dx \nonumber\\
	&   \sim 
\exp(  n^p f\left( {x_0} \right)) \sqrt {\frac{{2\pi }}{{ -n^p t_1^2 t_2^2 f''\left( {x_0} \right)}}}.\label{eq:asy_proof_7}
\end{align}

This means that when $n$ is large enough,
\begin{align*} 
&
(1-\epsilon)\exp(  n^p f\left( {x_0} \right)) \sqrt {\frac{{2\pi }}{{ -n^p t_1^2 t_2^2 f''\left( {x_0} \right)}}}<
	A_3(n) \\&<(1+\epsilon) \exp(  n^p f\left( {x_0} \right)) \sqrt {\frac{{2\pi }}{{ -n^p t_1^2 t_2^2 f''\left( {x_0} \right)}}}.
\end{align*}

Therefore
\begin{align*} 
&
(1 - \epsilon)^4
n^{- \frac{p}{2}}
\exp(  n^p f\left( {x_0} \right)) \sqrt {\frac{{2\pi }}{{ - t_1^2 t_2^2 f''\left( {x_0} \right)}}} + o\left(n^{- \frac{p}{2}} \exp\left(  n^p f(x_0)   \right) \right)
\\&
	< \frac{ H_3(n)}{ g(x_0)A_0n^{-b_1-b_2- 1+ {p}}}
\\&
<
(1 + \epsilon)^4
n^{- \frac{p}{2}}
\exp(  n^p f\left( {x_0} \right)) \sqrt {\frac{{2\pi }}{{ - t_1^2 t_2^2 f''\left( {x_0} \right)}}} + o\left(n^{- \frac{p}{2}} \exp\left(  n^p f(x_0)   \right) \right).
\end{align*}
Finally we obtain
\begin{align*}
	d_n&= H_1(n)+H_2(n)+H_3(n)  	\sim H_3(n)\\
	& \sim g(x_0)A_0n^{-b_1-b_2- 1+ \frac{p}{2}}
\exp(  n^p f\left( {x_0} \right)) \sqrt {\frac{{2\pi }}{{ - t_1^2 t_2^2 f''\left( {x_0} \right)}}}.
\end{align*}

But 
$$
g(x_0)=(\frac{r_1}{r_1+r_2})^{-b_1-1+\frac{p}{2}}(\frac{r_2}{r_1+r_2})^{-b_2-1+\frac{p}{2}},
$$
$$
f(x_0)=(r_1+r_2)^{1-p},
$$
$$
f''(x_0)=\frac{p(p-1)
(r_1+r_2)^{3-p}}
{r_1r_2}.
$$

Therefore 
\begin{equation*}
d_n \sim 
\psi_{n}(v_1v_2, r_1+r_2, b_1+b_2; p).\qedhere
\end{equation*}
\end{proof}

Our  asymptotic formula can be easily iterated 
for handling a product of multiple power series $F_1(q)F_2(q)\cdots F_k(q)$.

\begin{thm}\label{th:asy_multi}
Suppose that $m>0,$ $1\leq i\leq m,$ $t=\gcd(z_1,z_2,\ldots, z_m).$ 
Let 
$$
F_i(q)= \sum_{n=0}^\infty a_n^{(i)} q^{z_in}
$$
and 
$$
G(q)=\prod_{i=1}^m F_i(q)=\sum_{n=0}^\infty d_{n} q^{tn}
$$
where
\begin{align}\label{eq:asy_multi_1}
a_n^{(i)} \sim z_i\cdot\psi_{z_in}( v_i,  r_i, b_i; p).
\end{align}
Then
\begin{align}\label{eq:asy_multi_2}
d_{n} \sim t\cdot\psi_{tn}(\prod_{i=1}^m v_i, \sum_{i=1}^m r_i,\sum_{i=1}^m b_i; p).
\end{align}
\end{thm}
\begin{proof}
Without loss of generality, we can assume that $m=2$.
For $i=1,2,$ we have 
$$
F_i(q)= \sum_{n=0}^\infty a_n^{(i)} (q^t)^{z_in/t}
$$
where
\begin{align*}
a_n^{(i)} \sim z_i\cdot\psi_{z_in}( v_i,  r_i, b_i; p) = 
\frac{z_i}{t}\psi_{z_in/t}(v_it^{-\frac{b_i}{1-p}}, r_it^{\frac{p}{1-p}}, b_i; p).
\end{align*}
Replace $q$ by $q^t$, $t_i$ by $\frac{z_i}{t}$ in Theorem \ref{th:asy_main2} we have
\begin{align*}
	d_{n} &\sim \psi_{n}(v_1v_2t^{-\frac{b_1+b_2}{1-p}}, (r_1+r_2)t^{\frac{p}{1-p}}, b_1+b_2; p) \\
	&=t\psi_{tn}(v_1v_2, r_1+r_2, b_1+b_2; p). \qedhere
\end{align*}
\end{proof}

Hardy-Ramanujan \cite{HardyRamanujan1918} have discovered the asymptotic formula for the number of integer partitions, which was extended by Ingham \cite{Ingham1941} in $1941$.
\begin{lem}[Ingham]\label{th:Asy_Ribbon:2}
Let $x$ and $y$ be two positive integers with $\gcd(x,y)=1$. Suppose that 
$$
\prod_{k\geq 0}\frac {1}{1-q^{xk+y}}= \sum_{n=0}^\infty a_n q^n.
$$
Then
$$
a_n\sim \psi_{n}(v, r, b; \frac12),
$$
where
$$
v=\frac{\Gamma(y/x)}{\sqrt{x\pi}} (\frac {x}{2})^{y/x}, \qquad
r=\frac {2\pi^2}{3x}, \qquad 
b=\frac{y}{2x} - \frac 14.
$$
\end{lem}

Ingham's result can be further generalized as follows, which will be useful
for finding the asymptotic formula for skew plane partitions.

\begin{thm}\label{th:Asy_Ribbon:multi_general}
Suppose that $m>0,$  $x_i>0,y_i>0,$ $z_i=\gcd(x_i,y_i)$ for $1\leq i\leq m$,  $t=\gcd(z_1,z_2,\ldots, z_m)$ and $$
\qquad
v_i=\frac{\Gamma(y_i/x_i)}{\sqrt{x_i\pi}} (\frac {x_i}{2})^{y_i/x_i}, \qquad
r_i=\frac {2\pi^2}{3x_i}, \qquad 
b_i=\frac{y_i}{2x_i} - \frac 14.
$$   
Let
$$
\prod_{i=1}^m \prod_{k\geq 0}\frac {1}{1-q^{x_ik+y_i}}=\sum_{n=0}^\infty d_{n} q^{tn}.
$$
Then
\begin{align}\label{eq:asy_product_2}
d_{n} \sim t\cdot\psi_{tn}(\prod_{i=1}^m v_i, \sum_{i=1}^m r_i,\sum_{i=1}^m b_i; \frac12).
\end{align}
\end{thm}
\begin{proof}
	Let
$$
	\prod_{k\geq 0}\frac {1}{1-q^{x_ik+y_i}}= \sum_{n=0}^\infty a_n^{(i)} q^{z_in}.
$$
It is easy to check that 
$$
 \psi_{n}(v_i z_i^{\frac12-\frac{y_i}{x_i}}, r_i z_i, b_i; \frac12) = z_i\cdot\psi_{z_i n}( v_i,  r_i, b_i; \frac12).
$$
Replace $q$ by $q^{z_i},$ $x$ by $x_i/z_i,$ $y$ by $y_i/z_i$ in Lemma \ref{th:Asy_Ribbon:2} we obtain 
$$
a_n^{(i)}\sim \psi_{n}(v_i z_i^{\frac12-\frac{y_i}{x_i}}, r_i z_i, b_i; \frac12) \sim z_i\cdot\psi_{z_i n}( v_i,  r_i, b_i; \frac12).
$$
Thus \eqref{eq:asy_product_2} follows from Theorem \ref{th:asy_multi}.
\end{proof}

The above result implies Theorem \ref{th:Asy_Ribbon:multi} by letting $t=1$.
Next we give the proof of Theorem \ref{th:p0_asy_main2}.

\begin{proof}[Proof of Theorem \ref{th:p0_asy_main2}]
By induction, it is easy to see that we just need to prove the case $m=1$, $t_1=t$. 
Notice that $(x^\alpha \exp(\beta x^p))'=(\beta px^p+\alpha)x^{\alpha-1}
	\exp(\beta x^p)$.
Let $0<\epsilon<1$. Then there exists some $N>0$ such that for any $x\geq N$, we have $(x^\alpha \exp(\beta x^p))'>0$; and
for any $n\geq N$, we have
\begin{align}\label{eq:p0_asy_proof_2}
(1 - \epsilon)n^\alpha \exp(\beta n^p)< a_{n} < (1 + \epsilon)n^\alpha
	\exp(\beta n^p).
\end{align}

But $$d_n=\sum_{j=0}^{\floor*{\frac{n}{t}}}a_{n-tj}=\sum_{j=0}^{\floor*{\frac{n-N}{t}}}a_{n-tj}+\sum_{j=\floor*{\frac{n-N}{t}}+1}^{\floor*{\frac{n}{t}}}a_{n-tj}.$$

First we have 
\begin{align*}
&
\sum_{j=\floor*{\frac{n-N}{t}}+1}^{\floor*{\frac{n}{t}}}a_{n-tj}=O(1).
\end{align*}
On the other hand,   we have
\begin{align}
&
(1 - \epsilon)\sum_{j=0}^{\floor*{\frac{n-N}{t}}}(n-tj)^\alpha \exp(\beta(n-tj)^p)< \sum_{j=0}^{\floor*{\frac{n-N}{t}}}a_{n-tj}
\\&
< (1 + \epsilon)\sum_{j=0}^{\floor*{\frac{n-N}{t}}}(n-tj)^\alpha \exp(\beta(n-tj)^p).
\end{align}

Since $x^\alpha \exp(\beta x^p)$ increases for $x\geq N$, we have

\begin{align*}
&
\frac1t\int_{n-t\floor*{\frac{n-N-t}{t}}}^{n}  x^\alpha \exp(\beta x^p) dx
\leq
\sum_{j=0}^{\floor*{\frac{n-N}{t}}}(n-tj)^\alpha \exp(\beta(n-tj)^p)
\\&
\leq \frac1t\int_{n-t\floor*{\frac{n-N}{t}}}^{n+t}  x^\alpha \exp(\beta x^p) dx.
\end{align*}

But when $n$ is large enough, we have 
\begin{align*}
\int_{n-t\floor*{\frac{n-N}{t}}}^{n+t}  x^\alpha \exp(\beta x^p) dx
&\sim \int_{n-t\floor*{\frac{n-N}{t}}}^{n+t}  \Bigl(x^\alpha+\frac{\alpha+1-p}{\beta p}x^{\alpha-p}\Bigr) \exp(\beta x^p) dx
\\&=
\int_{n-t\floor*{\frac{n-N}{t}}}^{n+t}  \bigl(\frac{x^{\alpha+1-p}}{\beta p} \exp(\beta x^p)\bigr)' dx\\
	& \sim \frac{(n+t)^{\alpha+1-p}}{\beta p} \exp(\beta (n+t)^p)\\
& \sim \frac{n^{\alpha+1-p}}{\beta p} \exp(\beta n^p),
\end{align*}
where the last $\sim$ is guaranteed by the condition $0<p<1$.

Similarly, we have 
\begin{align*}
\int_{n-t\floor*{\frac{n-N-t}{t}}}^{n}  x^\alpha \exp(\beta x^p) dx
&\sim  \frac{n^{\alpha+1-p}}{\beta p} \exp(\beta n^p).
\end{align*}

Therefore when $n$ is large enough, we have

$$
(1 - \epsilon)^2\frac{n^{\alpha+1-p}}{\beta pt} \exp(\beta n^p)< \sum_{j=0}^{\floor*{\frac{n-N}{t}}}a_{n-tj} < (1 + \epsilon)^2\frac{n^{\alpha+1-p}}{\beta pt} \exp(\beta n^p).
$$

Finally we obtain
\begin{equation*}
d_n \sim \frac{n^{\alpha+1-p}}{\beta pt} \exp(\beta n^p).\qedhere
\end{equation*}

\end{proof}


\section{Asymptotic formulas for skew plane partitions}\label{sec:PP} 
Various plane partitions have been widely studied since MacMahon \cite{MacMahon1899,MacMahon1916}.
In particular, the generating function for skew plane partitions with profile 
$\delta$ has been derived (see \cite{Borodin2007,  Sagan1993, Stanley1972}). 
In this section, first we obtain the asymptotic formula
for ordinary plane partitions of fixed width. We say that a (skew) plane partition $\omega$ has a width $m$ if $\omega_{i,j}=0$ for $i>m$.

\begin{thm}\label{th:AsymPPmn}
	Let $\PP{PP}_m(n)$ be the number of 
 plane partitions $\omega$ of width $m$ and size $n$. Then, 
	\begin{equation}\label{eq:AsymPPmn}
	\PP{PP}_m(n) \ \sim \  
	2^{-\frac{m^2+2m+5}{4}}  (\frac{m}{3})^{ \frac {m^2+1}4} \pi^\frac{m^2-m}{2}  \prod_{i=1}^{m-1} i!\ \times   n^{-\frac {m^2+3}4} 
\exp(\pi\sqrt{2mn/3}).
	\end{equation}
\end{thm}
\begin{proof}
Let $\delta= 1^m (-1)^\infty$ in \eqref{eq:gfSkewPP} we have 
$$
\sum_{\omega \in \PP{PP}_m} z^{|\omega|}
=
\prod_{k\geq 0}\prod_{i=1}^m\frac{1}{(1-z^{k+i})}.
$$
Therefore by Theorem \ref{th:Asy_Ribbon:multi} the number of plane partitions with profile $\delta$ and  size $n$ is asymptotic to
\begin{equation*}
\psi_{n}(\prod_{i=1}^m \frac{(i-1)!}{\sqrt{\pi}} (\frac {1}{2})^{i},\,
	\frac {2m\pi^2}{3},\, \frac {m^2}4;\, \frac12),
\end{equation*}
	which is equal to the right-hand side of \eqref{eq:AsymPPmn}.
\end{proof}
When $m=1$, the above theorem gives the Hardy-Ramanujan asymptotic formula 
for the number of integer partitions
	\begin{equation}
	\PP{PP}_1(n) \ \sim \  
		\frac{1}{4\sqrt{3}n} \exp(\pi\sqrt{2n/3}).
	\end{equation}
When $m=3$, this is the example (PPa)  
illustrated in Fig.~2. 
Actually we have 
$$
\sum_{\omega \in \PP{PPa}} z^{|\omega|}
=
\prod_{k\geq 0}\frac{1}{(1-z^{k+1})(1-z^{k+2})(1-z^{k+3})}.
$$
Therefore the number of plane partitions of width $3$ and  size $n$ is asymptotic to
$$
2^{-4} \pi^3 n^{-3}\exp(\pi\sqrt{2n}).
$$

\medskip

More generally, we can derive the asymptotic formula for the number of skew plane 
partitions with fixed width.

\begin{thm}\label{th:AsymSkPPmn}
Let $\delta=(\delta',(-1)^\infty)=(\delta'_1,\delta'_2,\ldots,\delta'_{m-1},(-1)^\infty)$ be a profile, 	and $\PP{SkewPP}_\delta(n)$ be the number of 
skew plane partitions  with profile $\delta$  and size $n$. 
 Then 
	\begin{align}
	\PP{SkewPP}_\delta(n) \ &\sim \  
		2^{-\frac {\ell^2+2\ell+5}{4}}(\frac {\ell}{3})^{\frac {\ell^2+1}{4}}  \pi^{\frac{\ell^2-\ell}{2}} 
\prod_{\substack{ i<j   \\ \delta'_i > \delta'_j   } } 
\frac{1}{j-i} \prod_{\delta'_{i}=1} (m-i-1)!  \nonumber\\
		&\qquad \times 
n^{-\frac {\ell^2+3}{4}  } 
		\exp(\pi\sqrt{2\ell n/3}),\label{eq:Asym_SkewPP}
	\end{align}
	where $\ell:=|\delta'|_1$.
\end{thm}
\begin{proof}
By \eqref{eq:gfSkewPP} we have 
$$
\sum_{\omega \in \PP{SkewPP}_\delta} z^{|\omega|}
=
\prod_{\substack{ i<j   \\ \delta'_i > \delta'_j   } } 
\frac{1}{1-z^{j-i}}  \times \prod_{k\geq 0}\prod_{\delta'_{i}=1}\frac{1}{1-z^{k+m-i}}.
$$ 
By Theorem \ref{th:Asy_Ribbon:multi} the coefficient of $z^n$ in   
$$
\prod_{k\geq 0}\prod_{\delta'_{i}=1}\frac{1}{1-z^{k+m-i}}
$$
is asymptotic to
$$
\psi_{n}(\prod_{\delta'_{i}=1} \frac{(m-i-1)!}{\sqrt{\pi}} (\frac {1}{2})^{m-i},\,
	\frac {2\ell\pi^2}{3},\,
	\frac {1}2\sum_{\delta'_{i}=1}(m-i)-\frac{\ell}{4};\, \frac12).
$$
Denote by $\inv({\delta'})$ the number of pairs $(i,j)$ 
such that $ i<j, \delta'_i > \delta'_j $. 
	Notice that $ \sum_{\delta'_{i}=1}(m-i)- \inv({\delta'})=\binom{\ell+1}{2}$. 
Then by Theorem \ref{th:p0_asy_main2} the number $\PP{SkewPP}_\delta(n)$ is asymptotic to
\begin{align*}
&	2^{-3/2+\inv({\delta'})}\pi^{-\frac 12} (\frac {2\ell\pi^2}{3})^{\frac {\ell^2+1}{4}}
\prod_{\substack{ i<j   \\ \delta'_i > \delta'_j   } } 
\frac{1}{j-i} \prod_{\delta'_{i}=1} \frac{(m-i-1)!}{\sqrt{\pi}} 
	(\frac {1}{2})^{m-i}
\\&\times
n^{-\frac {\ell^2+3}{4}  } 
\exp(\pi\sqrt{2\ell n/3}),
\end{align*}
which is equal to the right-hand side of \eqref{eq:Asym_SkewPP}.
\end{proof}

The two examples (PPb) and (PPc) for $\ell=|\delta'|_1=3$ 
illustrated in Fig.~2 correspond to the following special cases 
of Theorem \ref{th:AsymSkPPmn}.

(PPb) Let $\delta= (1,1,-1,1,  (-1)^\infty)$ we have 
\begin{align*}
\sum_{\omega \in \PP{PPb}} z^{|\omega|}
&=
\frac{1}{1-z}\frac{1}{1-z^2}\prod_{k\geq 0}\frac{1}{(1-z^{k+1})(1-z^{k+3})(1-z^{k+4})}
.
\end{align*}
Therefore by Theorems \ref{th:Asy_Ribbon:multi} and \ref{th:p0_asy_main2} the number of skew plane partitions with profile $\delta$ and size $n$ is asymptotic to
$$
3\cdot 2^{-4} \pi^3 n^{-3}\exp(\pi\sqrt{2n}).
$$

(PPc) Let $\delta= (1,1,-1,-1,1,  (-1)^\infty)$ we have 
\begin{align*}
\sum_{\omega \in \PP{PPc}} z^{|\omega|}
&=
\frac{1}{1-z}\frac{1}{1-z^2}\frac{1}{1-z^2}\frac{1}{1-z^3}\prod_{k\geq 0}\frac{1}{(1-z^{k+1})(1-z^{k+4})(1-z^{k+5})}
.
\end{align*}
Therefore the number of skew plane partitions with profile $\delta$ and  size $n$ is asymptotic to
$$
3\cdot 2^{-3} \pi^3 n^{-3}\exp(\pi\sqrt{2n}).
$$


\section{Asymptotic formula for cylindric partitions}\label{sec:CP} 
First we recall Borodin's formula \eqref{eq:gfCP} written in the following form.

\begin{lem}[Borodin \cite{Borodin2007}]\label{th:Borodin}
Let $\delta=(\delta_i)_{1\leq i\leq h}$ be a profile. 
Then the generating function for the cylindric partitions with profile $\delta$ is
\begin{equation*}
\sum_{\omega \in \PP{CP}_\delta}  z^{|\omega|}
=  
	\prod_{k\geq 0}  
	\prod_{t\in W_\delta} \frac{1}{1-z^{hk+t}},
\end{equation*}
where
\begin{align*}
	W_\delta&=\{h\} \cup \{ j-i : i < j,\ \delta_i > \delta_j\} \cup \{ h+i-j: i < j,\ \delta_i < \delta_j \}. 
\end{align*}
\end{lem}

In this section we derive the asymptotic formula for the number of cylindric partitions.
\begin{thm}\label{th:CP_ASY}
Let 
$\delta=(\delta_j)_{1\leq j\leq h}$ be a profile. When $1\leq |\delta|_1 \leq h-1$,
the number of cylindric partitions with profile $\delta$ is asymptotic to
\begin{equation*}
 \prod_{\substack{i < j \\ \delta_i > \delta_j}}  
	\Gamma(\frac{j-i}{h}) 
\prod_{\substack{i < j \\ \delta_i < \delta_j}} \Gamma(\frac{h+i-j}{h}) \, \,
	\frac{\sqrt{1+2K}}{4\sqrt {3}\cdot (2\pi)^K} 
   \times
\frac{1}{n}
\exp\left({\pi\sqrt{\frac{2(1+2K)n}{3h}}}\right),
\end{equation*}
	where $K=|\delta|_1 |\delta|_{-1}/2$.
\end{thm}

\begin{proof}
We have 
\begin{align}
\sum_{t\in W_\delta} \frac{t}{h} 
&=
\frac{h}{h}+ \sum_{i < j,\ \delta_i > \delta_j} \frac{j-i}{h} + \sum_{i < j,\ \delta_i < \delta_j} \frac{h+i-j}{h}  \nonumber
\\
	& =
	1+ \sum_{\substack{ i<j   \\ \delta_i < \delta_j   } } 1 + \frac{|\delta|_1}{h}   \sum_{\delta_i=-1} i - \frac{|\delta|_{-1}}{h}\sum_{\delta_j=1} j.
	\label{eq:sumth}
\end{align}
If we exchange any two adjacent letters in $\delta$, the right-hand side of
	\eqref{eq:sumth} doesn't change. 
Therefore,
the above summation is independent of $\delta$ when $h$ and $|\delta|_1$ are given.  
Hence we have
\begin{align*}
\sum_{t\in W_\delta} \frac{t}{h} 
&=
1+ K.
\end{align*}
On the other hand, $\#W_\delta= 1+2K$. Then we obtain
	\begin{equation}\label{eq:1/4}
\sum_{t\in W_\delta} (\frac{t}{2h} - \frac 14) = \frac 14.
	\end{equation}
By Lemma \ref{th:Borodin} ,  Theorem \ref{th:Asy_Ribbon:multi} and
	Identity \eqref{eq:1/4} the number 
	of cylindric partitions with profile $\delta$ and size $n$ is asymptotic to
$$
\psi_{n}(v, r, b; \frac12),
$$
where  
\begin{align*}
v&= \prod_{t\in W_\delta} \left(\frac{\Gamma(t/h)}{\sqrt{h\pi}} (\frac {h}{2})^{t/h} \right)  
=\prod_{t\in W_\delta} \Gamma(\frac{t}{h})\cdot 2^{-1-K} \sqrt{h} 
	\pi^{-\frac12-K}, \\
r&= \sum_{t\in W_\delta } \frac {2\pi^2}{3h} 
=\frac {2\pi^2}{3h} (1+2K),\\
b&= \sum_{t\in W_\delta} (\frac{t}{2h} - \frac 14) = \frac 14.
\end{align*}
	The proof is achieved by the definition \eqref{def:psi} of $\psi$.
\end{proof}

Notice that in the above theorem, the profile $\delta$ contains both steps ``$1$" and ``$-1$".
In fact, 
when $\delta=(-1)^h$ or  $\delta=(1^h)$,
the number of cylindric partitions with profile $\delta$  is $0$ if 
$n$ is not a multiple of $h$.
If $n=hn_1$, it is  equal to  the number of integer partitions of size $n_1$.

\medskip

The three examples (CPa)-(CPc) for $h=4$ (here we say that these cylindric partitions have width h=4)
illustrated in Fig.~2 correspond to the following special cases 
of Theorem \ref{th:CP_ASY}.

(CPa) Let $\delta=  (1，-1,-1,-1)$. By Lemma \ref{th:Borodin} we have 
$$
\sum_{\omega \in \PP{CPa}} z^{|\omega|}
=
\prod_{k\geq 0}\frac{1}{(1-z^{4k+1})(1-z^{4k+2})(1-z^{4k+3})(1-z^{4k+4})}=\prod_{k\geq 0}\frac{1}{1-z^{k+1}}.
$$
Therefore the number of such cylindric partitions with size $n$ is asymptotic to
$$
\frac{\sqrt{3}}{12}
\times  \frac{1}{n} 
	\exp\Bigl(\pi\sqrt{\frac{2n}{3}} \, \Bigr).
$$

(CPb) Let $\delta=  (1，-1,1,-1)$. By Lemma \ref{th:Borodin} we have 
\begin{align*}
\sum_{\omega \in \PP{CPb}} z^{|\omega|}
&=
\prod_{k\geq 0}\frac{1}{(1-z^{4k+1})^2(1-z^{4k+3})^2(1-z^{4k+4})}
\\&=
\prod_{k\geq 0}\frac{1}{(1-z^{2k+1})^2(1-z^{4k+4})}.
\end{align*}
Therefore the number of such cylindric partitions with size $n$ is asymptotic to
\begin{align*}
\frac18\sqrt{\frac53}
\times  \frac{1}{n} 
	\exp\Bigl(\pi\sqrt{\frac{5n}{6}} \, \Bigr).
\end{align*}

(CPc) Let $\delta=  (1，-1,-1,1)$. By Lemma \ref{th:Borodin} we have 
\begin{align*}
\sum_{\omega \in \PP{CPc}} z^{|\omega|}
&=
\prod_{k\geq 0}\frac{1}{(1-z^{4k+1})(1-z^{4k+2})^2(1-z^{4k+3})^2(1-z^{4k+4})}
\\&=
\prod_{k\geq 0}\frac{1}{(1-z^{k+1})(1-z^{4k+2})}.
\end{align*}
Therefore the number of such cylindric partitions with size $n$ is asymptotic to
$$
\frac18\sqrt{\frac56}
\times  \frac{1}{n} 
	\exp\Bigl(\pi\sqrt{\frac{5n}{6}} \, \Bigr).
$$


\bibliographystyle{plain}
\bibliography{x14.bib}

\begin{thebibliography}{10}

\bibitem{Borodin2007}
A.~Borodin.
\newblock Periodic {S}chur process and cylindric partitions.
\newblock {\em Duke Math. J.}, 140(3):391--468, 2007.

\bibitem{DewarMurty2013}
M.~Dewar and M.~R. Murty.
\newblock An asymptotic formula for the coefficients of {$j(z)$}.
\newblock {\em Int. J. Number Theory}, 9(3):641--652, 2013.

\bibitem{Erdelyi1956}
A.~Erd\'elyi.
\newblock {\em Asymptotic expansions}.
\newblock Dover Publications, Inc., New York, 1956.

\bibitem{GesselKratt1997}
I.~M. Gessel and C.~Krattenthaler.
\newblock Cylindric partitions.
\newblock {\em Trans. Amer. Math. Soc.}, 349(2):429--479, 1997.

\bibitem{HardyRamanujan1918}
G.~H. Hardy and S.~Ramanujan.
\newblock Asymptotic {F}ormula\ae \ in {C}ombinatory {A}nalysis.
\newblock {\em Proc. London Math. Soc.}, S2-17(1):75, 1918.

\bibitem{Ingham1941}
A.~E. Ingham.
\newblock A {T}auberian theorem for partitions.
\newblock {\em Ann. of Math. (2)}, 42:1075--1090, 1941.

\bibitem{Kotesovec2015}
V.~Kot\v{e}\v{s}ovec.
\newblock A method of finding the asymptotics of q-series based on the
  convolution of generating functions.
\newblock { arXiv:1509.08708}, 2015.

\bibitem{Langer2013B}
R.~Langer.
\newblock Enumeration of cylindric plane partitions -- part {II}.
\newblock { arXiv:1209.1807}, 2012.

\bibitem{Macdonald1995}
I.~G. Macdonald.
\newblock {\em Symmetric functions and {H}all polynomials}.
\newblock Oxford Mathematical Monographs. The Clarendon Press, Oxford
  University Press, New York, second edition, 1995.
\newblock With contributions by A. Zelevinsky, Oxford Science Publications.

\bibitem{MacMahon1899}
P.~A. MacMahon.
\newblock Partitions of numbers whose graphs possess symmetry.
\newblock {\em Trans. Cambridge Philos. Soc.}, 17:149--170, 1899.

\bibitem{MacMahon1916}
P.~A. MacMahon.
\newblock {\em Combinatory analysis. {V}ol. {I}, {II} (bound in one volume)}.
\newblock Dover Phoenix Editions. Dover Publications, Inc., Mineola, NY, 2004.
\newblock Reprint of {\it An introduction to combinatory analysis} (1920) and
  {\it Combinatory analysis. Vol. I, II} (1915, 1916).

\bibitem{Sagan1993}
B.~E. Sagan.
\newblock Combinatorial proofs of hook generating functions for skew plane
  partitions.
\newblock {\em Theoret. Comput. Sci.}, 117(1-2):273--287, 1993.
\newblock Conference on Formal Power Series and Algebraic Combinatorics
  (Bordeaux, 1991).

\bibitem{Stanley1971}
R.~P. Stanley.
\newblock Theory and application of plane partitions. {I}, {II}.
\newblock {\em Studies in Appl. Math.}, 50:167--188; 259--279, 1971.

\bibitem{Stanley1972}
R.~P. Stanley.
\newblock {\em Ordered structures and partitions}.
\newblock American Mathematical Society, Providence, R. I., 1972.
\newblock Memoirs of the American Mathematical Society, No. 119.

\bibitem{StanleyEC2}
R.~P. Stanley.
\newblock {\em Enumerative combinatorics. {V}ol. 2}, volume~62 of {\em
  Cambridge Studies in Advanced Mathematics}.
\newblock Cambridge University Press, Cambridge, 1999.
\newblock With a foreword by Gian-Carlo Rota and appendix 1 by Sergey Fomin.

\bibitem{Tingley2008}
P.~Tingley.
\newblock Three combinatorial models for {$\widehat{\rm sl}_n$} crystals, with
  applications to cylindric plane partitions.
\newblock {\em Int. Math. Res. Not. IMRN}, 2008(2):Art. ID rnm143, 40, 2008.

\end{thebibliography}

\end{document}